\documentclass[11pt]{amsart}
\usepackage{amscd}
\usepackage{amsmath}
\usepackage{amsxtra}
\usepackage{amsfonts}
\usepackage{amssymb}
\usepackage{mathabx}

\usepackage{color}
\usepackage{cite}

\newtheorem{theorem}{Theorem}[section]

\newtheorem{Lem}[theorem]{Lemma}
\newtheorem{Prop}[theorem]{Proposition}

\newtheorem{conjecture}[theorem]{Conjecture}
\theoremstyle{definition}
\newtheorem{Def}[theorem]{Definition}

\newtheorem{example}[theorem]{Example}
\theoremstyle{remark}

\renewcommand{\theclaim}{\textup{\theclaim}}

\numberwithin{equation}{section}

\def\openone

{\mathchoice

{\hbox{\upshape \small1\kern-3.3pt\normalsize1}}

{\hbox{\upshape \small1\kern-3.3pt\normalsize1}}

{\hbox{\upshape \tiny1\kern-2.3pt\SMALL1}}

{\hbox{\upshape \Tiny1\kern-2pt\tiny1}}}

\makeatletter

\newbox\ipbox

\newcommand{\diracb}[1]{\left\langle #1\mathrel{\mathchoice

{\setbox\ipbox=\hbox{$\displaystyle \left\langle\mathstrut
#1\right.$}

\vrule height\ht\ipbox width0.25pt depth\dp\ipbox}

{\setbox\ipbox=\hbox{$\textstyle \left\langle\mathstrut
#1\right.$}

\vrule height\ht\ipbox width0.25pt depth\dp\ipbox}

{\setbox\ipbox=\hbox{$\scriptstyle \left\langle\mathstrut
#1\right.$}

\vrule height\ht\ipbox width0.25pt depth\dp\ipbox}

{\setbox\ipbox=\hbox{$\scriptscriptstyle \left\langle\mathstrut
#1\right.$}

\vrule height\ht\ipbox width0.25pt depth\dp\ipbox}

}\right. }

\newcommand{\dirack}[1]{\left. \mathrel{\mathchoice

{\setbox\ipbox=\hbox{$\displaystyle \left.\mathstrut
#1\right\rangle$}

\vrule height\ht\ipbox width0.25pt depth\dp\ipbox}

{\setbox\ipbox=\hbox{$\textstyle \left.\mathstrut
#1\right\rangle$}

\vrule height\ht\ipbox width0.25pt depth\dp\ipbox}

{\setbox\ipbox=\hbox{$\scriptstyle \left.\mathstrut
#1\right\rangle$}

\vrule height\ht\ipbox width0.25pt depth\dp\ipbox}

{\setbox\ipbox=\hbox{$\scriptscriptstyle \left.\mathstrut
#1\right\rangle$}

\vrule height\ht\ipbox width0.25pt depth\dp\ipbox}

} #1\right\rangle}

\newcommand{\B}{\mathcal{B}}
\newcommand{\beq}{\begin{equation}}
\newcommand{\eeq}{\end{equation}}

\def\blfootnote{\xdef\@thefnmark{}\@footnotetext}

\input xy
\xyoption{all}
\usepackage{amssymb}

\def\R{\mathbb{R}}
\def\N{\mathbb{N}}

\def\A{\mathcal{A}}
\def\E{\mathcal E}

\def\Q{\mathbb{Q}}

\def\-{^{-1}}
\def\B{\mathcal{B}}
\def\D{\mathcal{D}}

\def\C{\mathbb{C}}
\def\Z{\mathbb{Z}}
\def\A{\mathcal{A}}

\title{Classification of spectral self-similar measures with four-digit elements}

\author{Lixiang An}
\address{[Lixiang An] School of Mathematics and Statistics,
$\&$ Hubei Key Laboratory of Mathematical Sciences,
Central China Normal University,
Wuhan 430079,
P.R. China.}

 \email{anlixianghai@163.com}

\author{Xinggang He}

\address{[Xinggang He] School of Mathematics and Statistics,
$\&$ Hubei Key Laboratory of Mathematical Sciences,
Central China Normal University,
Wuhan 430079,
P.R. China.}
 \email{xingganghe@163.com}

\author{Chun-Kit Lai}

\address{[Chun-Kit Lai] Department of Mathematics, San Francisco State University,
1600 Holloway Avenue, San Francisco, CA 94132.}

 \email{cklai@sfsu.edu}

\date{}

\thanks{The research of Lixiang An and Xinggang He is supported by NSFC grant 12171181 and 11971194.}

\begin{document}

\begin{abstract}
    Let $\mu$ be a self-similar measure generated by iterated function system of four maps of equal contraction ratio $0<\rho<1$. We study when $\mu$ is a spectral measure which means that it admits an exponential orthonormal basis $\{e^{2\pi i \lambda x}\}_{\lambda\in\Lambda}$ in $L^2(\mu)$. By combining  previous results of many authors and a careful study of some new cases, we completely classify all spectral self-similar measures with four maps (Theorem \ref{main-theorem}). Moreover, the case allows us to propose a modified {\L}aba-Wang conjecture concerning when the self-similar measures are spectral in general cases.  
\end{abstract}

\maketitle
\begin{center}
{\it Dedicated to the memory of Professor Ka-Sing Lau}
\end{center}

\section{introduction}
Let $\mu$ be a Borel probability measure and let $\Lambda$ be a countable set in ${\Bbb R}$ and denote by
 $E(\Lambda)=\{e_\lambda: \lambda\in\Lambda\}$ where
 $e_\lambda(x)=e^{2\pi i\lambda x}$. We say that $\Lambda$ is a {\it spectrum} if $E(\Lambda)$ is an orthonormal basis for $L^2(\mu)$. If $L^2(\mu)$ admits a spectrum, then $\mu$ is called a {\it spectral measure} and we will sometimes call $(\mu, \Lambda)$ a {\it spectral pair}. A natural question in Fourier analysis would be
 
 \medskip
 
 {\bf (Qu):} Which Borel probability measure $\mu$ is a spectral measure? 
 
 \medskip
 
 Thanks to the result in \cite{HLL2013}, a spectral measure must be of pure type in the sense that it is either purely absolutely continuous, purely discrete or purely singular with respect to Lebesgue measure without any atoms. 
 
 \medskip
 
 In \cite{DL2014}, we know that an absolutely continuous spectral measure must be  the normalized Lebesgue measure supported on a measurable set $\Omega$ with finite positive measure. Hence, classifying absolutely continuous spectral measure reduced to the infamous Fuglede's conjecture \cite{F1974} which used to propose that $L^2(\Omega)$ admits a spectrum if and only if $\Omega$ is a translational tile. The conjecture is now known to be false in its full generality \cite{T2004,KM1,KM2} on dimension 3 or higher, but the conjecture was recently shown to be true if $\Omega$ is a convex domain \cite{LM2022}. This conjecture remains open in dimension 1 and 2 and the precise classification of spectral sets  remains as a highly interesting problem for the many researchers.

 \medskip
 
 In the purely discrete case, it is not too difficult to show that a discrete measure is spectral implies that it has only finitely many atoms and all the atom are of equal weight. i.e. $\mu = \delta_{\A} = \frac1{\#\A} \sum_{a\in\A}\delta_a$ where $\delta_a$ is the Dirac measure at $a$. Most progress has been  made for $\A$ having more additive structure. As a counter-example of the Fuglede's conjecture on finite groups immmediately leads to a counterexample in the classical Fuglede's conjecture, a lot of effort has been spent on studying the case where $\A$ is a subset of a finite group in which the spectrum is also taken from the same group. We refer readers to see \cite{LL2022, M2022} and the reference therein for some recent advance.

 \medskip
 
 Spectral singular measures without any atoms were first discovered by Jorgensen and Pedersen \cite{JP1998}. They discovered that the middle-fourth Cantor measure is a spectral measure while the middle-third Cantor measure is not spectral. There has been many new spectral singular measures discovered since then. The construction is first obtained from a digit set $\D$ whose Dirac measure $\delta_{\D}$ is a spectral measure with a spectrum ${\mathcal L}$ is a finite group $\Z^d/R(\Z^d)$ where $R$ is an integral expanding matrix (i.e. all eigenvalues has moduii larger than 1). We call $(R,\D,{\mathcal L})$  a {\it Hadamard triple}. Then we use $R$ and $\D$ to generate an iterated function system, $\{\phi_d(x) = R^{-1}(x+d):d\in\D\}$. The  {\it self-affine measure} of the iterated function system is the unique invariant measure $\mu = \mu_{R,\D}$ satisfying the invariant equation 
 $$
 \mu (E) = \sum_{d\in\D}\frac{1}{\#\D} \mu (\phi_d^{-1}(E)), \ \forall E \ \mbox{Borel}.
 $$
 Dutkay, Hausermann and the third-named author showed that all the above measures are spectral measures \cite{DHL2019}, while the case on $\R^1$ was first settled by {\L}aba and Wang \cite{LW2002}. The key reason for these self-affine measures to be spectral is that $\mu_{R,\D}$ is an infinite convolution of discrete measures
 $$
 \mu_{R,\D} = \delta_{R^{-1}\D}\ast\delta_{R^{-2}\D}\ast\delta_{R^{-3}\D}\ast...
 $$
 The spectrality of the first $n$ convolution are all be determined, and certain limiting arguments are investigated to determine the spectrality of its weak $-{\ast}$ limit $\mu_{R,\D}$.  Building on this prototypical result, other singular spectral measures were also studied by randomly convoluting discrete measures generated by different Hadamard triples (see e.g. \cite{AFL2019,LMW2022}).   This makes up one of the main categories of singular spectral measures. For other singular spectral measures, one may see \cite{LLP2021} in which surface measures on polytope are also spectral. 
 
 \medskip
 
\subsection{A review of the {\L}aba-Wang conjecture.}  Because we need to satisfy the rigid orthogonal equation in order to form spectral measures, it is particularly difficult for infinite convolution of discrete measures.  It has been asked if it is possible to have a classification of spectral self-similar measures. Recall that if we are given $0<\rho<1$, a digit set $\D = \{0,d_1,...,d_{m-1}\}$ and a probability vector ${\bf p}  = (p_0,...,p_{m-1)}$,  consider the iterated function system (IFS) $\phi_j(x) = \rho (x+d_j)$, the self-similar measure $\mu = \mu_{\rho,\D,{\bf p}}$ is the unique probability measure satisfying the self-similar identity
 $$
 \mu(E) = \sum_{j=0}^{m-1} p_j \mu (\phi_j^{-1}(E)).
 $$
This is the conjecture {\L}aba and Wang proposed \cite{LW2002}.
 
 \begin{conjecture}
Suppose that the self-similar measure $\mu_{\rho,\D,{\bf p}}$ is a spectral measure. Then 
\begin{enumerate}
    \item ${\bf p}$ is a equal-weight probability vector. i.e. $p_j = 1/m.$
    \item $\rho = 1/N$ for some integer $N$.
    \item There exists $\alpha>0$ such that $\D = \alpha {\mathcal C}$ where ${\mathcal C}\subset \Z^+$ and ${\mathcal C}$ tiles $\Z_N$. i.e. there exists $\B$ such that ${\mathcal C}\oplus\B\equiv \Z_N$ (mod N).
\end{enumerate}
 \end{conjecture}
 The above direct sum means that $c+b$ are all distinct elements for all $c\in{\mathcal C}$ and $b\in\B$ and they form a complete residue classes modulo $N$. Condition (i) has been completely answered by Deng and Chen \cite{DC2021}.  For earlier results, please see \cite{DL2014}.
 
 \begin{theorem}[Deng and Chen, \cite{DC2021}]\label{Theorem_DC}
Let $\mu_{\rho,\D,{\bf p}}$ be a self-similar measure with $0<\rho<1$, $\D\subset\R$ and ${\bf p}$ is a probability vector.  Suppose that $\mu_{\rho,\D,{\bf p}}$ is a spectral measure. Then ${\bf p}$ must be a equal-weighted probability vector. 
\end{theorem}
 
 We will from now on denote $\mu_{\rho,\D}$ to be the self-similar measure with equal-weighted probability vector. There are also satisfactory answer to condition (ii) by the first-named author and Wang \cite{AW2021}, whose results were attributed to  a series of earlier studies of the spectrality of the Bernoulli convolution and $N$-Bernoulli measures by Dai, He, Hu and Lau \cite{HL2008,D2012,DHL2014}.

 \begin{theorem}[An and Wang,\cite{AW2021}]\label{lem8}
Let $\mu_{\rho, \D}$ be a self-similar with $0<\rho<1$ and $\D\subset\R$.  Suppose that there exists $\alpha>0$ such that ${\mathcal Z}(M_{\D})\subset {\alpha}{\Bbb Z}$. If $\mu_{\rho, \D}$ is a spectral measure, then $\rho^{-1}=N\in \N$.
\end{theorem}
 
\medskip

 Condition (iii) were much more delicate to deal with since it is related to Fuglede's conjecture on cyclic groups. To avoid relating the problem to Fuglede's conjecture, one may impose another condition which will be equivalent to (iii) if Fuglede's conjecture on cyclic groups were proven.
 
 \medskip
 
 (iii)':  There exists ${\mathcal L}\subset \Z$ such that $(N,\D,{\mathcal L})$ forms a Hadamard triple.
 
 \medskip
 
\noindent Condition (iii)' is unfortunately not the correct condition to state either. A simple example, first observed by Dutkay and Jorgensen \cite{DJ2009}, already disproved (iii). Let $N = 4$ and $\D = \{0,1,8,9\}$. Then the self-similar measure $\mu_{4^{-1},\D}$ is the Lebesgue measure on $[0,1]\cup[2,3]$ which is spectral with a spectrum $\Z+\{0,1/4\}$. This is the case where the self-similar measure is absolutely continuous. In \cite{DL2014}, Dutkay and Lai proved that if a spectral self-similar measure is absolutely continuous, then {\L}aba-Wang conjecture can be solved with (iii) replaced by $\D$ tiles some cyclic groups, not necessarily $\Z_N$ where $1/N$ is the contraction ratio. Indeed, the measure must be the Lebesgue measure supported on the  self-similar tiles generated by the IFS. 

\medskip

The authors studied this kind of digit sets and introduced the product-form Hadamard triple in the previous paper \cite{AL2022}.

\begin{Def}\label{definition-product-form}
We say that $(N, \D,{\mathcal L}_1\oplus{\mathcal L}_2)$ forms a {\it product-form Hadamard triple} if there exists $\A$ and $\B_a$ for each $a\in \A$ such that 
$$
\D = \bigcup_{a\in\A} (a+ N\B_a)
$$
where 
\begin{enumerate}
    \item $(N, \A,{\mathcal L}_1)$ and $(N, \B_a,{\mathcal L}_2)$ are Hadamard triples for each $a\in\A$.
    \item $(N, \A\oplus\B_a,{\mathcal L}\oplus{\mathcal L}_2)$ are Hadamard triples for each $a\in\A$.
\end{enumerate}
\end{Def}

\medskip

It is direct to show that the counter-example is now a special case of the product-form Hadamard triple. The main result of our previous paper was to show that 

\begin{theorem}
The self-similar measure $\mu_{N,\D}$ is a spectral measure if $(N, \D,{\mathcal L}_1\oplus{\mathcal L}_2)$ forms a product-form Hadamard triple for some ${\mathcal L_1}, {\mathcal L}_2\subset \Z$.
\end{theorem}
 
 The previous paper also considered some higher stage product-form Hadamard triple. However, they will be reduced to Definition \ref{definition-product-form} by considering some higher power.  In view of this, we propose the following modified {\L}aba-Wang conjecture.
 
 \medskip
 
  \begin{conjecture} (modified {\L}aba-Wang conjecture)
Suppose that the self-similar measure $\mu_{\rho,\D,{\bf p}}$ is a spectral measure and $\D\subset \Z^{+}\cup\{0\}$ with $0\in\D$. Then 
\begin{enumerate}
    \item ${\bf p}$ is a equal-weight probability vector. i.e. $p_j = 1/m.$
    \item $\rho = 1/N$ for some integer $N\ge 2$.
    \item There exists $\alpha>0$ such that $\D = \alpha {\mathcal C}$ and there exists $m>0$ and $k>0$ such that if we define 
    $$
    {\bf D}_k = (m{\mathcal C})+N(m{\mathcal C})+...+ N^{k-1}(m{\mathcal C}),
    $$
    then $(N^k, {\bf D}_k, {\mathcal L}_1\oplus{\mathcal L}_2)$ forms a product-form Hadamard triple for some ${\mathcal L}_1, {\mathcal L}_2$. 
\end{enumerate}
 \end{conjecture}
 
 \medskip
 
 \subsection{Main Results.} This paper aims to study the modified {\L}aba-Wang conjecture is some special case based on the cardinality of  $\D$. When $\#\D = 2$, the self-similar measure is known as the Bernoulli convolution, and a complete answer has been given by Dai \cite{D2012}. When $\#\D = 3$, the case for product-form Hadamard triples will not appear still and it can be classified easily. For the sake of completeness, we will give a full proof based on the theorems we quoted so far for $\#\D = 2$ and $3$ and show that they satisfy the {\L}aba-Wang conjecture in Section 5. The first non-trivial case is $\#\D = 4$. The main body of the paper will be to prove the following theorem which provides the complete classification of self-similar spectral measures when $\#\D = 4$.

 \begin{theorem}\label{main-theorem}
 Let $\D \subset \R^+\cup\{0\}$ and $0\in\D$ be a digit set with $\#\D = 4$, $0<\rho<1$ and ${\bf p}$ be a probability vector. Suppose that the self-similar measure $\mu_{\rho,\D,{\bf p}}$ is a spectral measure. Then
 \begin{enumerate}
     \item ${\bf p}$ is a equal-weighted probability vector.
     \item $\rho  = \frac{1}{N}$ where $N = 2^{\beta} m$ for some integer $\beta\ge 1$ and $m$ is an odd integer. 
     \item By rescaling, $\D = \{0, a, 2^t\ell, a+2^{t}\ell'\}$ where $a,\ell,\ell'$ are odd integers and $t$ is not a multiple of $\beta$. 
 \end{enumerate}
 In particular, if we write $t = \beta k+r$ for some $0<r\le \beta-1$, then $(N, m^k\D,{\mathcal L}\oplus{\mathcal L}_2)$ is a product-form Hadamard triple with ${\mathcal L}_1 = \{0,2\}$ and ${\mathcal L}_2 = \{0, N 2^{r-1}\}.$ Hence, the modified {\L}aba-Wang conjecture holds.  
 \end{theorem}
 
\noindent{\bf Outline of the Proof of Theorem \ref{main-theorem}.} The proof of Theorem \ref{main-theorem} will be divided into several steps which will be completed in several sections. The proof will be following the outline as below and readers can refer to the relevant sections for the detailed proofs:
 \begin{enumerate}
     \item[(1)]  Statement (i) has been known by Deng and Chen Theorem (Theorem \ref{Theorem_DC}). We will assume ${\bf p}$ is equal-weighted and write the measure as $\mu_{\rho,\D}$.
     \item[(2)] In Section 3, we will prove that $\D$ must be rational.  To do that, by translation and rescaling, we will write $\D = \{0,1,a,b\}$ where $1<a<b$. In Theorem \ref{theorem irr}, we will prove that if $\mu_{\rho,\D}$ is spectral. Then $a,b$ must be  rational numbers.
     \item[(3)] Starting from Section 4, we assume that digits in $\D$ are all rational. Hence, by rescaling, we will write $\D = \{0,a,b,c\}$ where $a<b<c$ are integers. By studying the zero set, we can rewrite 
     $$
     \D = \{0,a,2^{t_1}\ell,a+2^{t_2}\ell'\},
     $$
     where $a,\ell,\ell'$ are odd integers and $t_1,t_2\ge 1$. 
     In Subsection \ref{subsection 4.2}, we will prove that (ii) holds in Lemma \ref{main-lem1}. Then we will prove that  (iii) holds by showing that for $\mu_{\rho,\D}$ to be spectral, $t_1 = t_2$ and they are not divisible by $\beta$ are necessary (Theorem \ref{main theo}). 
     \item [(4)]The last part of the theorem were proven in \cite[Theorem 1.11]{AL2022}. For the sake of completenees, we provide the proof in Subsection \ref{subsection 4.1}
    \end{enumerate}

\subsection{Notation.} Let us setup the main notations we will use in this paper. Since we will very often use the set of all odd integers, we denote it as ${\mathbb O}$.  Given a Borel probability measure $\mu$,
 the Fourier transform is defined to be
 $$
 \widehat{\mu}(\xi) = \int e^{-2\pi i \xi x}d\mu(x).
 $$
 The zero set of $\widehat{\mu}$ will be denoted by ${\mathcal Z}(\widehat{\mu})$ and the same notation ${\mathcal Z}(f)$ for the zero set of a continuous function $f$. $\mu\ast \nu$ will denote the convolution of $\mu$ and $\nu$. We will also write 
  $$
 \Asterisk_{i=1}^r \mu_i = \mu_1\ast\mu_2\ast...\ast\mu_r
 $$
 to denote convolutions of $r$ measures. For a given finite digit set $\D$, the $\delta_{\D}$ denotes the equal-weighted Dirac measure supported on $\D$ {\it mask function} is the Fourier transform of $\delta_{\D}$, i.e. 
$$
M_{\D}(\xi) = \frac1{\#\D}\sum_{d\in\D} e^{-2\pi i d\xi}.
$$
For self-similar measures $\mu_{\rho,\D}$, we know that we can write 
 \begin{equation}\label{mu_conv}
 \mu_{\rho,\D} = \Asterisk_{i=1}^{\infty} \delta_{\rho^i\D}: = \mu_n\ast \mu_{>n}    
 \end{equation}
 where $\mu_n$ is the convolution of the first $n$ discrete measures and $\mu_{>n}$ is the rest of the convolutions
 
 \medskip
 
 Concerning the orthogonal sets, we will always assume $0\in\Lambda$.  $E(\Lambda)$ is an orthogonal set for $L^2(\mu)$ if and only if
\begin{equation}\label{bizero}
(\Lambda-\Lambda)\setminus\{0\}\subset {\mathcal Z}(\widehat{\mu}) := \{\xi\in{\Bbb R}: \widehat{\mu}(\xi)=0\}.
\end{equation}
We call the set $\Lambda$ satisfying (\ref{bizero}) a {\it bi-zero set} of $\mu$. As $0$ is in any bi-zero set, so that $\Lambda\subset\Lambda-\Lambda$. 
 Given a countable set $\Lambda$ and a Borel probability measure $\mu$, let 
$$
Q(\xi):=\sum_{\lambda\in\Lambda}|\widehat{\mu}(\xi+\lambda)|^2.
$$
The Jorgensen-Pedersen lemma \cite{JP1998} implies that $\Lambda$ is a spectrum for $\mu$ if and only if $Q(\xi) = 1$ for all $\xi\in\R$. $\Lambda$ is a bi-zero set of $\mu$ if and only if $Q(\xi)\le1$ for $\xi\in\R$. This fact will be used throughout the paper.

\section{Preliminaries}

\subsection{Basic Setup} In this section, we will set up some preliminary observation that reduces the problem. Let $0<\rho<1$. Since rescaling and translation will not affect the spectrality of the self-similar measure, namely $\mu_{\rho,\D}$ is spectral if and only if $\mu_{\rho,\alpha\D+t}$ is spectral for all $\alpha\ne 0$ and $t\in\R$,  we can assume that our digit  set has the following form
 $$
 \D = \{0,1,a,b\}, \ 1<a<b.
 $$
Then the mask function of $\D$ is given by 
$$
M_{\mathcal D}(\xi) = \frac14(1+e^{-2\pi i \xi}+e^{-2\pi i a\xi}+e^{-2\pi i b\xi}).
$$
A basic but fundamental fact about vanishing sum of complex numbers of modulus one is that given $z_1,z_2,z_3\in\C$ such that $|z_1| = |z_2| = |z_3| = 1$.  
$$
1+ z_1+z_2+z_3 = 0 
$$
if and only if one of complex numbers is $-1$ and the others are opposite sign of each other, e.g. $z_1 = -1$ and $z_2 +z_3 = 0$ (or permuting the indices 1,2,3). This can be proved by directly solving the complex equation. This fact leads us to the following lemma.  

\medskip

\begin{Lem}\label{Lem1.1}
Consider the digit set ${\mathcal D}$ as above. Then the zero set ${\mathcal Z}(M_{\mathcal D})\neq\emptyset$ if and only if at least one of the following ratios
$$
\frac{1}{b-a}, \ \frac{a}{b-1}, \ \frac{b}{a-1}
$$
is a rational number of the form $\frac{p}{q}$ with $p,q$ are odd numbers.
\end{Lem}

\medskip

\begin{proof}
Note that $M_{\mathcal D}(\xi)=0$ if and only if one of the exponentials is $-1$ and the remaining two is of opposite sign to each other:
\begin{equation}\label{eq3equations}
\left\{
  \begin{array}{ll}
    e^{2\pi i \xi} = -1,  \\
    e^{2\pi i (b-a)\xi} =-1, &
  \end{array}
\right. \
\left\{
  \begin{array}{ll}
    e^{2\pi i a\xi} = -1,  \\
    e^{2\pi i (b-1)\xi} =-1, &
  \end{array}
\right. \
\left\{
  \begin{array}{ll}
    e^{2\pi i b\xi} = -1,  \\
    e^{2\pi i (a-1)\xi} =-1. &
  \end{array}
\right.
\end{equation}
In the first case, there are solutions for $\xi$ if and only if  $ \frac12(2k+1)=\frac{1}{2(b-a)}(2\ell+1)$ which implies $\frac{1}{b-a}$ is rational of the  desired form. The same argument applies for the other two cases.
\end{proof}

We now use the the digit set $\D$ to generate the equal-weighted self-similar measure, denoted by $\mu_{\rho, \D}$ with $0<\rho<1$. The self-similar measure can be written as 
$$
\mu_{\rho,\D} = \delta_{\rho\D}\ast\delta_{\rho^2\D}\ast....
$$
Hence, the Fourier transform of $\mu_{\rho,\D}$ is given by
$$
\widehat{\mu}(\xi) = \prod_{k=1}^{\infty}M_{\D}(\rho^k\xi).
$$
Thus, we have 
\begin{equation}\label{EQ_Z}
{\mathcal Z}(\widehat{\mu}) = \bigcup_{k=1}^{\infty}\rho^{-k}{\mathcal Z}(M_{\mathcal D}).
\end{equation}

\begin{Lem}\label{Lem1.2}
With respect to the three situations in Lemma \ref{Lem1.1}, we have
\medskip

(i) If $\frac{1}{b-a} = \frac{p}{q}$, with $p,q$ are odd, then
$${\mathcal Z}(M_{\mathcal D}) \supseteq \frac{p}{2}{\mathbb O}, \ {\mathcal Z}(\widehat{\mu}_{\rho, \D})\supseteq\bigcup_{k=1}^{\infty}\frac{\rho^{-k}p}{2}{\mathbb O}.
$$

\medskip

(ii) If $\frac{a}{b-1} = \frac{p}{q}$, with $p,q$ are odd,
then
$${\mathcal Z}(M_{\mathcal D}) \supseteq \frac{p}{2a}{\mathbb O}, \ {\mathcal Z}(\widehat{\mu}_{\rho, \D})\supseteq\bigcup_{k=1}^{\infty}\frac{\rho^{-k}p}{2a}{\mathbb O}.
$$

\medskip

(iii) If  $\frac{b}{a-1} = \frac{p}{q}$, with $p,q$ are odd, then
$${\mathcal Z}(M_{\mathcal D}) {\supseteq \frac{p}{2b}{\mathbb O}}, \ {\mathcal Z}(\widehat{\mu}_{\rho, \D})\supseteq{\bigcup_{k=1}^{\infty}\frac{\rho^{-k}p}{2b}{\mathbb O}}.
$$

\medskip

Suppose furthermore that only one of the three ratios $\frac{1}{b-a},\frac{a}{b-1}, \frac{b}{a-1}$ is rational. Then  $\supseteq$ can be replaced by $=$ for that case.
\end{Lem}

\begin{proof}

We only prove the first case. The others are similar. Note that in the first case we can write
$$
{M_{\D}}(\xi) = \frac14\left(1+e^{2\pi i \xi}+e^{2\pi i a\xi}(e^{2\pi i (b-a)\xi}+1)\right)=  \frac14\left(1+e^{2\pi i \xi}+e^{2\pi i a\xi}(e^{2\pi i \frac{q}{p}\xi}+1)\right).
$$
As $p,q$ are odd, it is direct to check that $M_{\D}\left(\frac{p}{2}(2k+1)\right)=0$ for any integer $k$. Hence, we have ${\mathcal Z}(M_{\mathcal D}) \supseteq \frac{p}{2}{\mathbb O}$. The second inclusion follows from (\ref{EQ_Z}).

\medskip

We now prove the last statement. Suppose now that only $\frac{1}{b-a}$  is rational. Then the only possibility for  $\xi\in{\mathcal Z}(M_{\D})$ is to satisfy the first equation in (\ref{eq3equations}). i.e.  $\left\{
  \begin{array}{ll}
    e^{2\pi i \xi} = -1,  \\
    e^{2\pi i q/p\xi} =-1, &
  \end{array}
\right. \
$
Hence, $\xi= \frac{1}{2}(2m+1) = \frac{p}{2q}(2n+1)$. This implies that $p(2n+1) = q(2m+1)$. As $p,q$ are co-prime, $q$ divides $2n+1$ and thus we can write $\xi = \frac{p}{2}(2n'+1)$ for some integer $n'$. This establishes the opposite inclusion.
\end{proof}

\medskip

\subsection{Zero sets of multi-Bernoulli structure.} We will be dealing with zero set of $\widehat{\mu}$ having structure  resembles the zero set of the Bernoulli convolution. We say that a measure $\mu$ has a {\it zero set of multi-Bernoulli structure} if
$$
{\mathcal Z}(\widehat{\mu})= \bigcup_{i=1}^{s}\alpha_i\left\{\frac{\rho^{-k}a}{2}: a\in{\mathbb O}, k\geq1\right\}
$$
for some $0<\rho<1$ and distinct $\alpha_i>0$.
We say that $\mu$ has a {\it zero set of Bernoulli structure} if $s=1$, which means the zero set of $\mu$ is a multiple of a zero set of Bernoulli convolutions of contraction ratio $\rho$.
Extracting the main conclusion of \cite{HL2008}, we have

\medskip

\begin{theorem}[Hu and Lau]\label{lem7}
Suppose that $\mu$ has a zero set of Bernoulli structure for some prime integer $p$. Then it has an infinite  bi-zero set if and only if $\rho=\left(\frac{n}{m}\right)^{1/r}$ where $r\geq1$, $m$ is even  and $n$ is odd.
\end{theorem}

\medskip


When dealing with measures with convolution structure, the following simple lemma will often be used. We omit the proof and reader can see \cite[Lemma 2.2]{DHL2014} for a proof.

\begin{Lem}\label{lem5}
Let $\mu=\mu_0\ast...\ast\mu_k$ be the convolution of the probability measures $\mu_i, i=0,1,...,k $, and suppose that none of the  $\mu_i$ is a Dirac measure of one point. Suppose that $\Lambda$ is a bi-zero set of the measure $\mu_i$ for some $i$, then $\Lambda$ is also a bi-zero set of the measure $\mu$ but it is not a spectrum for $\mu$.
\end{Lem}

We now consider measures $\mu$ with zero set of multi-Bernoulli structure.  Proposition \ref{prop2.1} below will allow us to focus on $\rho $ being a root of fractions. The proof require a version of infinite Ramsey theorem about coloring of graphs. One can refer to \cite[Theorem A]{R1929} for details. Suppose that $S$ is a countably infinite set and we let $S(n)$ be the collection of sets of $S$ whose cardinality is $n$.

\begin{theorem}\label{Ramsey}
[Infinite Ramsey Theorem] Let $X$ be a countably infinite set.  Then for all $n, s\geq2$ and for every partition of the set $X(n)$ into $s$ classes, one of the classes contains every element of $Y(n)$ for some infinite set $Y\subset X$.
\end{theorem}

\begin{Prop}\label{prop2.1}
Let $\mu$ be a probability measure in $\R$. Suppose that $\mu$ has a zero set of multi-Bernoulli structure.  ${\mathcal Z}(\widehat{\mu})$ contains an infinite bi-zero set if and only if $\rho=\left(\frac{n}{m}\right)^{1/r}$ where $r\geq1$,  $m$ is even  and $n$ is odd.
\end{Prop}

\begin{proof}
Let ${\mathcal Z}(\widehat{\mu})= \bigcup_{i=1}^{s}\alpha_i\{\frac{\rho^{-k}a}{2}: a\in{\mathbb O}, k\geq1\}.$ If $\rho=\left(\frac{n}{m}\right)^{1/r}$ where $r\geq1$,  $m$ is even  and $n$ is odd, according to  Theorem \ref{lem7}, there is an infinite bi-zero set in $\alpha_i\{\frac{\rho^{-k}a}{2}: a\in{\mathbb O}, k\geq1\}$ which is also an bi-zero set of $\mu$. The sufficiency follows.

\bigskip

 Suppose that $\rho\not\in\{\left(\frac{n}{m}\right)^{1/r}: r\geq1, m \text{ is even  and } n \text{ is odd}\}$ and there is an infinite bi-zero set $\Lambda$. We partition all subsets of $\Lambda(2)$  into $s$ classes:
 $$
T_i=\left\{\{\lambda_1, \lambda_2\}\subset\Lambda: \lambda_1-\lambda_2\in\alpha_i\{\frac{\rho^{-k}a}{2}: a\in{\mathbb O}, k\geq1\}\right\}, \ i = 1,2,...,s.
$$
 By the infinite Ramsey theorem, there is an  infinite subset $\Gamma$ of $\Lambda$ such that $\Gamma(2)$ are in one of the $T_i$. It means that 
 $$(\Gamma-\Gamma)\setminus\{0\}\subset \alpha_i\left\{\frac{\rho^{-k}a}{2}: a\in{\mathbb O}, k\geq1\right\},$$ 
 which contradicts to Theorem \ref{lem7}. Hence $\Lambda$ is finite.
\end{proof}

\medskip

Before we go into the proof, we also need a notion of maximal bizero set. 

\begin{Def} \label{Maximal bizero} 
Suppose that a Borel probability measure is given by a convolution $\mu = \nu_1\ast\nu_2$ and $\Lambda$ is a bi-zero set of $\mu$. We say that a subset ${\mathcal A}$ is a {\it maximal bi-zero set of $\Lambda$ in $\nu_1$} if ${\mathcal A}$ is a bi-zero set in $\nu_1$ and whenever $\lambda\in\Lambda\setminus {\mathcal A}$, there exists $\alpha\in{\mathcal A}$ such that $\lambda-\alpha\in {\mathcal Z}(\widehat{\mu})\setminus  {\mathcal Z}(\widehat{\nu_1})$.
\end{Def}

We notice that as long as ${\mathcal Z}(\widehat{\nu_1})$ is non-empty and $\lambda\in \Lambda\cap {\mathcal Z}(\widehat{\nu_1})\ne\emptyset$, $\{0,\lambda\}$ must be a bi-zero set of $\nu_1$. Hence, the maximal bi-zero set must have cardinality at least 2. We will use the notion of maximality several parts of  the arguments in the next section.

\medskip


\section{Irrational digits}

In this section, we will consider the digit set as in the last section. 
\begin{equation}\label{D}
\D = \{0,1,a,b\}, \  b>a>1.
\end{equation}
 Our main conclusion is the following:

\begin{theorem}\label{theorem irr}
Suppose that the self-similar measure $\mu_{\rho,\D}$ is a spectral measure with $\D$ as in (\ref{D}), then both $a$ and $b$ are rational.
\end{theorem}

\begin{proof}
In order for the measure $\mu_{\rho,\D}$ to be a spectral measure, its zero set and hence ${\mathcal Z}(M_{\D})$ must be non-empty. By Lemma \ref{Lem1.1}, we see that if ${\mathcal Z}(M_{\D})$ is non-empty and one of the $a,b$ is rational, then the other will also be rational. Hence, to prove Theorem \ref{theorem irr}, we may assume that {\it both $a$ and $b$ are irrational} and we show that the self-similar measures $\mu_{\rho,\D}$ cannot be a spectral measure.

\medskip

We divide our consideration into two parts according to Lemma \ref{Lem1.1} and they will be discussed in the next two cases.

\medskip

(1) Exactly one of the three ratios  $\frac{1}{b-a}$, $\frac{a}{b-1}$, $\frac{b}{a-1}$  is rational.

(2) At least two of the ratios are rational.

\medskip

\bigskip

\noindent{\bf Case (1): \underline{Exactly one of the three ratios   is rational.}}

\medskip

We only consider the case $\frac{1}{b-a}$ is rational, the other two is similar. We let that $\frac{1}{b-a}=\frac{p}{q}$ with co-prime odd integer $p, q$. Then ${\mathcal Z}(M_{\D})=\frac{p}{2}{\mathbb O}$ and ${\mathcal Z}(\widehat{\mu}_{\rho, \D})=\bigcup_{k=1}^{\infty}\frac{\rho^{-k}p}{2}{\mathbb O}.$  By  An-Wang Theorem (Theorem \ref{lem8}), $\rho = 1/m$ for some integer $m$. But ${\mathcal Z}(\widehat{\mu}_{\rho, \D})$ has a zero set of Bernoulli structure. We immediately know that it is necessary that  $\rho=\frac{1}{m}$ with $m$ is even. The following proposition  settles the final case.

\begin{Prop}\label{prop3.8}
Suppose $\rho=\frac{1}{m}$ with $m$ is even. Then $\mu_{\rho, \D}$ is not a spectral measure.
\end{Prop}
\begin{proof}

In this case, $${\mathcal Z}(\widehat{\delta}_{\rho\D})=\frac{mp}{2}{\mathbb O},\quad {\mathcal Z}(\widehat{\mu}_{\rho, \D})=\bigcup_{k=1}^{\infty}\frac{m^{k}p}{2}{\mathbb O}.$$
Suppose that  $\Lambda$ is a bi-zero set of $\mu_{\rho, \D}$. Recalling the convolutional formula (\ref{mu_conv}), we can set $\A\subset \Lambda$ to be a maximal bi-zero set of $\mu_{\rho,\D}$ in $\delta_{\rho\D}$.  Then
$$(\A-\A)\setminus\{0\}\subset\frac{mp}{2}{\mathbb O},$$
which implies that $\A$ is a bi-zero set of $\delta_{m^{-1}\{0,1\}}$ and $\delta_{m^{-1}\{0,p\}}$. For any $\alpha\in\A$, let
$$\Lambda_{\alpha}:=\Big\{\lambda\in\Lambda: \lambda-\alpha\in{\mathcal Z}(\widehat{\mu}_{\rho,\D})\setminus{\mathcal Z}(\widehat{\delta}_{\rho\D})\Big\}\cup\{\alpha\}.$$
Then
$$\Lambda_{\alpha}-\alpha\subset\{0\}\cup\left(\bigcup_{k=2}^{\infty}\frac{m^{k}p}{2}{\mathbb O}\right)\subset mp\Z,$$
and hence
$$(\Lambda_{\alpha}-\Lambda_{\alpha})\setminus\{0\}\subset\bigcup_{k=2}^{\infty}\frac{m^{k}p}{2}{\mathbb O}.$$
It implies that $m^{-1}\Lambda_{\alpha}$ is a bi-zero set of $\mu_{\rho,\D}$.

\begin{eqnarray}\label{eqprop3.2}
Q(\xi) &=& \sum_{\lambda\in\Lambda}|\widehat{\mu}_{\rho, \D}(\xi+\lambda)|^2
=\sum_{\lambda\in\Lambda}\left|M_{\D}(m^{-1}(\xi+\lambda))\right|^2|\widehat{\mu}_{\rho, \D}(m^{-1}(\xi+\lambda))|^2\nonumber\\
&=&\sum_{a\in\A}\sum_{\lambda\in \Lambda_{\alpha}}\left|M_{\D}(m^{-1}(\xi+\lambda))\right|^2\left|\widehat{\mu}(m^{-1}(\xi+\lambda))\right|^2\nonumber\\
\end{eqnarray}
Since $\frac{1}{b-a}=\frac{p}{q}$, we have
\begin{equation}\label{eqstrict}
\begin{aligned}
|M_{\D}(\xi)|^2=&\frac{1}{16}\left|(1+e^{2\pi i\xi}+e^{2\pi ia\xi}(1+e^{2\pi i q/p\xi}))\right|^2\\
\le&\frac{1}{4}\left(\left|\frac{1+e^{2\pi i\xi}}{2}\right|+\left|\frac{1+e^{2\pi i q/p\xi}}{2}\right|\right)^2\\
\le &\frac{1}{2}\left(\left|\frac{1+e^{2\pi i\xi}}{2}\right|^2+\left|\frac{1+e^{2\pi i q/p\xi}}{2}\right|^2 \right)\\
=& \frac12\left(|M_{\{0, 1\}}(\xi)|^2+|M_{\{0, \frac{q}{p}\}}(\xi)|^2\right).
\end{aligned}
\end{equation}
We have used $(x+y)^2\leq2(x^2+y^2)$ in the last inequality. The equality holds if and only if $x=y$. Note that
$$\Lambda_{\alpha}=\alpha+\Lambda_{\alpha}-\alpha\subset \alpha+mp\Z$$
and  $|M_{\{0, 1\}}(\xi)|$ is 1-period, $|M_{\{0, q/p\}}(\xi)|$ is $p/q$-period.
Putting these into (\ref{eqprop3.2}),

\begin{eqnarray*}\label{eqprop3.3}
&&Q(\xi)\le\sum_{a\in\A}\sum_{\lambda\in \Lambda_{\alpha}}\frac12\left(\left|M_{\{0,1\}}(m^{-1}(\xi+\lambda))\right|^2+\left|M_{\{0,q/p\}}(m^{-1}(\xi+\lambda))\right|^2\right)\left|\widehat{\mu}(m^{-1}(\xi+\lambda))\right|^2\\
&=&\sum_{a\in\A}\frac12\left(\left|M_{\{0,1\}}(m^{-1}(\xi+\alpha))\right|^2+\left|M_{\{0,q/p\}}(m^{-1}(\xi+\alpha))\right|^2\right)\sum_{\lambda\in \Lambda_{\alpha}}\left|\widehat{\mu}(m^{-1}(\xi+\lambda))\right|^2\\
&\le&\sum_{a\in\A}\frac12\left(\left|M_{\{0,1\}}(m^{-1}(\xi+\alpha))\right|^2+\left|M_{\{0,q/p\}}(m^{-1}(\xi+\alpha))\right|^2\right)\\
&\le&1.
\end{eqnarray*}
Let $0<\xi<\min\{\frac{1}{2}, \frac{p}{2q}\}$, then $0<\frac{2\pi \xi}{m}, \frac{2\pi q \xi}{mp}<\pi$, which together with $\frac{q}{p}\neq1$, yields that 
$$|1+e^{2\pi i\xi/m}|\neq|1+e^{2\pi i q\xi /mp}|.$$ Hence, for this $\xi$, $$|M_{\D}(m^{-1}\xi)|^2<\frac12\left(|M_{\{0, 1\}}(m^{-1}\xi)|^2+|M_{\{0, \frac{q}{p}\}}(m^{-1}\xi)|^2\right).$$ Hence, the inequality in the third line of  (\ref{eqstrict}) is indeed strict at $\alpha = 0$, proving the proposition.
\end{proof} 

\bigskip

\noindent{\bf Case (2): \underline{At least two of the ratios are rational.}}

\medskip

We start proving this case with a simple lemma.

\begin{Lem}\label{lem6}
Suppose that at least two of the ratios are rational. Then ${\mathcal D} = \{0,1,a,a+1\}$ with $a$ is irrational and
$$
{\mathcal Z}(M_{\D})=\frac{1}{2}{\mathbb O}\cup \frac{1}{2a}{\mathbb O}.
$$

\end{Lem}

\begin{proof}Note that

(i) If $ \frac{a}{b-1}=\frac{p_1}{q_1}$ and $ \frac{1}{b-a}=\frac{p_2}{q_2}$. Then
$$b=\frac{p_1}{q_1}a+1=a+\frac{p_2}{q_2} \ \Longrightarrow \  (1-\frac{p_1}{q_1})a=1-\frac{p_2}{q_2}.$$
Since $a$ is irrational, we have $\frac{p_1}{q_1}=\frac{p_2}{q_2}=1$. In this case, $\D=\{0, 1, a, a+1\}$.

(ii) If $ \frac{1}{b-a}=\frac{p_2}{q_2}$ and $\frac{b}{a-1}=\frac{p_3}{q_3}$. Then
$$b=a+\frac{p_2}{q_2}=\frac{p_3}{q_3}a-\frac{p_3}{q_3} \ \Longrightarrow \  \frac{p_2}{q_2}=-1,\frac{p_3}{q_3}=1.$$
Then $b=a-1$ which contradicts to $b>a$.

(iii) If $ \frac{a}{b-1}=\frac{p_1}{q_1}$ and $\frac{b}{a-1}=\frac{p_3}{q_3}$, then
$$b=\frac{p_1}{q_1}a+1=\frac{p_3}{q_3}a-\frac{p_3}{q_3} \ \Longrightarrow \  \frac{p_1}{q_1}=\frac{p_3}{q_3}=-1.$$
Then $b=-a+1<0$ which is a contradiction again.

\bigskip
This shows that $\D=\{0, 1, a, a+1\}$. 
As now the mask function of $\D$ is $M_{\D}(\xi) = \frac14(1+e^{2\pi i \xi})(1+e^{2\pi i a\xi})$, the ${\mathcal Z}(M_{\D})$ must have the desired structure.
\end{proof}

\medskip

As $a$ is irrational, ${\mathcal Z}(M_{\D})$ is not in a lattice. Theorem \ref{lem8} cannot be used. However, the zero set has a multi-Bernoulli structure with $s = 2$. Indeed, we have
$$
{\mathcal Z}(\widehat{\mu}_{\rho,\D})=\bigcup_{k=1}^{\infty}\frac{\rho^{-k}}{2}{\mathbb O}\cup\bigcup_{k=1}^{\infty}\frac{\rho^{-k}}{2a}{\mathbb O}.$$
As a result of Proposition \ref{prop2.1}, bi-zero set of infinite cardinality exist only when $\rho=\left(\frac{n}{m}\right)^{1/r}$ with $n$ odd and $m$ even.
Denote respectively by $\Gamma_1$ and $\Gamma_2$ the sets $\bigcup_{k=1}^{\infty}\frac{\rho^{-k}}{2}{\mathbb O}$ and  $\bigcup_{k=1}^{\infty}\frac{\rho^{-k}}{2a}{\mathbb O}$. We have
\begin{equation}\label{eq_conv_mu}
\mu_{\rho,\D}=\mu_{\rho,\{0,1\}}\ast\mu_{\rho,\{0,a\}}
\end{equation}
and ${\mathcal Z}(\widehat{\mu}_{\rho,\{0,1\}}) =\Gamma_1$, ${\mathcal Z}(\widehat{\mu}_{\rho,\{0,a\}}) =\Gamma_2$.    For $1\le i\le r$, let 
\begin{equation}\label{eq-Gamma}
\Gamma_1^{(i)}=\bigcup_{k=0}^{\infty}\frac{\rho^{-(kr+i)}}{2}{\mathbb O} \text{\ \ and\ \  }\Gamma_2^{(i)}=\bigcup_{k=0}^{\infty}\frac{\rho^{-(kr+i)}}{2a}{\mathbb O}.    
\end{equation}
Note that $\Gamma_1^{(i)}$ is the zero set of the self-similar measure $\mu_i=\mu_{\frac{n}{m},\frac{m}{n}\rho^i\{0,1\}}$ and $\Gamma_2^{(i)}$ is the zero set of the self-similar measure $\nu_i=\mu_{\frac{n}{m},\frac{m}{n}\rho^{i}\{0,a\}}$. Moreover, we now have $$\mu_{\rho,\D}=\mu_1\ast\cdots\ast\mu_{r}\ast\nu_1\ast\cdots\ast\nu_{r}.$$
We will now  prove Proposition \ref{prop2.2} and Proposition \ref{prop2.3}. Combining these two propositions, we will prove that there is no spectral measure in case (2).

\medskip

In the proofs, Lemma \ref{lem5}  will be used frequently.  Moreover, an algebraic fact will also be used:
\begin{equation}\label{alg fact}
a_0+a_1\rho+\cdots+a_{r-1}\rho^{r-1}=0, \  \ a_j\in{\Bbb Q} \Longrightarrow a_0=\cdots=a_{r-1}=0,
\end{equation}
which follows directly from the fact that the minimal polynomial of $\rho$ is of degree $r$.

\begin{Prop}\label{prop2.2}
Let $\rho=\left(\frac{n}{m}\right)^{1/r}$ with co-prime integer $m, n$ such that the minimal polynomial of $\rho=\left(\frac{n}{m}\right)^{1/r}$ in $\Q[x]$ is $x^r-\frac{n}{m}$. If (i) $r=1$ or  (ii) $r\geq2$ and $a\in\rho^{-i_0}\Q$ for some integer $i_0$ holds, then  $\mu_{\rho,\D}$ is not a spectral measure.
\end{Prop}

\begin{proof}
\noindent(i) Suppose that  $r=1$. $\rho$ is a rational number. It is clear that $\Gamma_1\subset\Q, \Gamma_2\subset\frac1{a}\Q$ where $a$ is irrational. Hence
$$(\Gamma_1-\Gamma_1)\cap \Gamma_2=\emptyset,\  (\Gamma_2-\Gamma_2)\cap \Gamma_1=\emptyset.$$
Suppose that $\Lambda$ is any bi-zero set of $\mu_{\rho,\D}$, the above two equalities force that all of the $\lambda_1,\lambda_2,\lambda_1-\lambda_2$ belongs to one of the $\Gamma_i$, $i=1,2$. That is to say $\Lambda$ is a bi-zero set of $\mu_{\rho,\{0,a\}}$ or $\mu_{\rho,\{0,1\}}$ and it is not a spectrum of $\mu_{\rho,\D}$ from Lemma \ref{lem5}. Hence $\mu_{\rho,\D}$ is not a spectral measure.

\bigskip
\noindent (ii) Suppose now that  $r\geq2$ and $a\in\rho^{-i_0}\Q$. Then $\Gamma_1^{(r)}, \Gamma_2^{(i_0)}\subset\Q$. Moreover, we have 
$\Gamma_1^{(j)}\subset\rho^{-j}\Q$  for all $1\le j\le r-1$ and
$$
\Gamma_2^{(i)}=\rho^{-(i-i_0)}\left(\bigcup_{k=0}^{\infty}\frac{\rho^{-kr}}{2a\rho^{i_0}}{\mathbb O}\right)\subset\rho^{-(i-i_0)}\Q, \quad\text{ for all } i\neq i_0.
$$
Let $\Lambda$ be any bi-zero set of $\mu$. The above inclusion shows that for all distinct elements $\lambda_1,\lambda_2\in\Lambda$, we can write
 $$\lambda_1=\rho^{-j_1}c_1,\   \lambda_2=\rho^{-j_2}c_2 \text{ and }\lambda_1-\lambda_2=\rho^{-j}c,$$
 where $0\le j_1, j_2, j\le r-1$ and $c_1, c_2, c$ is nonzero rational. We have the following equation holds:
 $$
 \rho^{-j_1}c_1-\rho^{-j_2}c_2=\rho^{-j}c.
 $$
 From the algebraic fact (\ref{alg fact}), $j = j_1 = j_2$.  It yields that 
 $$
 \Lambda-\Lambda\setminus \{0\}\subset \left\{\begin{array}{ll}
     \Gamma_1^{(r)}\cup\Gamma_2^{(i_0)} & \mbox{if} \  j = 0   \\
     \bigcup_{i=1}^{r-1}\Gamma_1^{(i)}\cup\bigcup_{i\neq i_0}\Gamma_2^{(i)} & \mbox{if}  \ 0<j\le r-1
 \end{array} \right.
 $$
In the first case,$\Lambda$ must be a bi-zero set of $\mu_{r}\ast\nu_{i_0}$, while in the second case, $\Lambda$ is a bizero set of  $\left(\Asterisk_{i=1}^{r-1}\mu_i\right)\ast\left(\Asterisk_{i\neq i_0}\nu_i\right)$. At least one of the measures are missing from (\ref{eq_conv_mu}). Hence $\Lambda$ cannot be a spectrum of $\mu$ by Lemma \ref{lem5}.

\end{proof}

\medskip

\bigskip

\begin{Prop}\label{prop2.3}
Let $n, m$ be co-prime and let $r>1$ such that the minimal polynomial of $\rho=\left(\frac{n}{m}\right)^{1/r}$ in $\Q[x]$ is $x^r-\frac{n}{m}$ and $a\not\in \rho^{i}\Q$ for any $i$. Then  $\mu_{\rho,\D}$ is not a spectral measure.
\end{Prop}

\begin{proof}
We will prove the result by showing that any bi-zero set $\Lambda$ is not spectrum of  $\mu_{\rho, \D}$.  When $\Lambda$ is finite, clearly it can not be a spectrum. Suppose $\Lambda$ is infinite, since bi-zero sets (spectra) are invariant
under translation, without loss of generality, we always assume that $0\in\Lambda$. We partition $\Lambda(2)$  into $2r$ classes:
 $$
\E_j^{(i)}=\left\{\{\lambda_1, \lambda_2\}\subset\Lambda: \lambda_1-\lambda_2\in\Gamma_j^{(i)}\right\},\quad 1\le i\le r, j=1, 2.
$$
 By the infinite Ramsey theorem (Theorem \ref{Ramsey}), there is an  infinite subset $\Lambda_0$ of $\Lambda$ such that every subset of $\Lambda_0(2)$  are in  $\E_1^{(i)}$ or $\E_2^{(i)}$ for some $i = 1,...,r$. It means that there exists $i_0$ such that 
 $$(\Lambda_0-\Lambda_0)\setminus\{0\}\subset \Gamma_1^{(i_0)}  \ \mbox{or} \ (\Lambda_0-\Lambda_0)\setminus\{0\}\subset\Gamma_2^{(i_0)}.$$ 
We will only work on the first case. The second case is similar.  Assume
 \begin{equation}\label{Lambda_0}
      (\Lambda_0-\Lambda_0)\setminus\{0\}\subset \Gamma_1^{(i_0)}\subset \rho^{-i_0}\Q.
 \end{equation}
and $\Lambda_0$ is the maximal bi-zero set of $\Lambda$ in $\mu_i$ as in Definition \ref{Maximal bizero}. 
 
\medskip
If  $\Lambda_0=\Lambda$, then  $\Lambda$ is a bi-zero set of $\mu_{i}$ and hence it cannot be a spectrum of $\mu$ by Lemma \ref{lem5}.  Suppose $\Lambda_0\not=\Lambda$ and choose  $\lambda\in \Lambda\setminus\Lambda_0$,  there is an element $\lambda_0\in\Lambda_0$ such that 
$$\lambda-\lambda_0\in \bigcup_{i\neq i_0}\Gamma_1^{(i)}\cup\bigcup_{i=1}^{r}\Gamma_2^{(i)}\subset\bigcup_{i\neq i_0}\rho^{-i}\Q\cup\bigcup_{i=1}^{r}a^{-1}\rho^{-i}\Q .$$  
It together with equation \eqref{Lambda_0} imply that 
\begin{eqnarray*}
\Lambda_0-\lambda&=&(\Lambda_0-\lambda_0)-(\lambda-\lambda_0)\\
&\subset&\left\{\rho^{-i_0}\Q-\left(\bigcup_{i\neq i_0}\rho^{-i}\Q\cup\bigcup_{i=1}^{r}a^{-1}\rho^{-i}\Q\right)\right\}\cap{\mathcal Z}(\widehat{\mu}_{\rho, \D})\\
\end{eqnarray*}
We claim that 
$$
\left\{\rho^{-i_0}\Q-\left(\bigcup_{i\neq i_0}\rho^{-i}\Q\cup\bigcup_{i=1}^{r}a^{-1}\rho^{-i}\Q\right)\right\}\cap \rho^{-i_0}\Q = \emptyset. 
$$
To see this, suppose that the intersection is non-empty. Then there exists $c_1,c_2,c_3\in \Q$ such that 
$$
\rho^{-i_0}c_1-\rho^{-i}c_2 = \rho^{-i_0}c_3 \ \mbox{or} \ \rho^{-i_0}c_1-\rho^{-i}a^{-1}c_2 = \rho^{-i_0}c_3.
$$
The first case is impossible since the minimal polynomial of $\rho$ is of degree $r$, but $i_0,i<r$. The second case is also impossible since this implies that $a\in \rho^{-(i-i_0)}\Q$. Hence, this justfies the claim. 

\medskip

However, $\Lambda_0-\lambda\subset {\mathcal Z}(\widehat{\mu}_{\rho,\D})$ holds, so 

$$\Lambda_0-\lambda\subset \bigcup_{i\neq i_0}\rho^{-i}\Q\cup\bigcup_{i=1}^{r}a^{-1}\rho^{-i}\Q.
$$ 
By the infinite Ramsey theorem, there is an infinite subset $\Lambda_1-\lambda\subset\Lambda_0-\lambda$ such that 
$$(\Lambda_1-\lambda)-(\Lambda-\lambda)\subset \rho^{-i_1}\Q \quad\text{ with } 1\le i_1\neq i_0\le r\quad \text{ or }\quad  a^{-1}\rho^{-i_1}\Q \text{ with } 1\le i_1\le r.$$
But $\Lambda_1-\Lambda_1 = (\Lambda_1-\lambda)-(\Lambda-\lambda)$ and it is a subset of $\Lambda_0-\Lambda_0$. By (\ref{Lambda_0}), this should be a subset of $\rho^{-i_0}\Q$.It implies that 
$$\Lambda_1-\Lambda_1\subset  (\rho^{-i_1}\Q)\cap (\rho^{-i_0}\Q), \ \mbox{or} \  \subset (a^{-1}\rho^{-i_1}\Q)\cap (\rho^{-i_0}\Q).
$$
But we know that $(\rho^{-i_1}\Q)\cap (\rho^{-i_0}\Q) = \{0\}$ by the fract that minimal polynomial of $\rho$ is degree $r$ and $(a^{-1}\rho^{-i_1}\Q)\cap (\rho^{-i_0}\Q) = \{0\}$ holds by the fact that $a\not\in \rho^{i}\Q$ for any $i$.
Both of them have $\{0\}$ as their intersection, meaning that $\Lambda_1$ has only one point. It contradicts to $\Lambda_1\subset\Lambda_0$ is an infinite set.  We have completed the proof.
\end{proof}

 We now know that the self-similar measures in both cases are not be spectral. Hence, this completes the proof of Theorem \ref{theorem irr}. 
\end{proof}

\medskip

%
%
%
%
%
\section{Rational digits}
We now consider the rational digits. By rescaling,  we may assume that  $\D=\{0, a, b, c\}\subset\Z^+$ with $\text{gcd}(a,b,c)=1$. In \cite{FHL2015}, the authors studied the spectrality of self-similar measures with integer digit sets of cardinality 4 and contraction ratio being $1/4$. They showed that self-similar measures are spectral in their case if and only if the measure is the Lebesgue meaure supported on the self-similar tiles. Now, we are working general contraction ratio. Nonetheless, the following lemma is important to us in this section. It can be found in \cite[Lemma 5.3 and Lemma 5.4]{FHL2015}.

\begin{Lem}\label{lem4.1}
 Let $\D=\{0, a, b, c\}\subset\Z^+$ with $gcd(a,b,c)=1$. Then the equation $M_{\D}(\xi)=0$ has a solution if and only if exactly two of the $a, b, c$ are odd and the remaining one is even.
 
 \medskip
 
 Without loss of generality, if we assume that $a<c\in{\mathbb O}$ and $b=2^t\ell$ with $t\geq1$ and $\ell\in{\mathbb O}$, then the solutions of the equation $M_{\D}(\xi)=0$ are given by:

(i)
 $${\mathcal Z}(M_{\D})=\frac{1}{2\mbox{\rm gcd}(a, c-b)}{\mathbb O}\cup \frac{1}{2\mbox{\rm gcd}(c, b-a)}{\mathbb O},$$
 if $c-a=2^{t'}\ell'$ for some $t'\neq t$ and $\ell'\in{\mathbb O}$.

\medskip

(ii) $${\mathcal Z}(M_{\D})=\frac{1}{2\mbox{\rm gcd}(a, c-b)}{\mathbb O}\cup \frac{1}{2\mbox{\rm gcd}(c, b-a)}{\mathbb O}\cup \frac{1}{2^{1+t}\mbox{\rm gcd}(\ell,\ell')}{\mathbb O},$$
if $c-a=2^{t}\ell'$ with $\ell'\in{\mathbb O}$.
\end{Lem}

\medskip
Because of the previous lemma, we can let  $b = 2^{t_1}\ell_1$, $c-a=2^{t_2}\ell_2$ where $t_1, t_2\geq1$ and $\ell_1, \ell_2$ are positive odd integers. Therefore, we can write
\begin{equation}\label{eqD_rational}
\D = \{0, a, 2^{t_1}\ell_1, a+2^{t_2}\ell_2\}, \ a,\ell_1,\ell_2\in{\mathbb O}.
\end{equation}
i.e. $b = 2^{t_1}\ell_1$ and $c = a+2^{t_2}\ell_2$. 
 Our main result in this section is 

\begin{theorem}\label{main theo}
$\mu_{\rho, \D}$ is a spectral measure if and only if $\rho=\frac{1}{2^{\beta}m}$ where $m$ is an odd integer and $t = t_1=t_2$ with $t$ is not divisible by $\beta$.
\end{theorem}

Combining with the results in the Section 3, this will complete the proof of Theorem \ref{main-theorem}.

\bigskip

\subsection{Proof of the sufficiency of Theorem \ref{main theo}.} \label{subsection 4.1}The sufficiency has been proved in our previous paper \cite[Theorem 1.11]{AL2022}. Since the proof is short, we provide it here for completeness. Indeed, let $N = 2^{\beta}m $ where $\beta\ge 1$ and $m$ is odd, and if  $\D = \{0,a,2^t\ell_1,a+2^t\ell_2\}$ where $a,\ell_1,\ell_2$ are odd numbers and $t$ is not divisible by $\beta$ as in the assumption, we can write $t = \beta k +r$ where $k\ge 0$ and $r \in\{1,...,\beta-1\}$. Now, we can write 
$$
\D = (\{0\}+ 2^{\beta k} \{0,2^r\ell_1\})\cup (\{a\}+ 2^{\beta k} \{0,2^r\ell_2\})
$$
Multiplying $m^k$ on both sides, we obtain
$$
m^k\D = (\{0\}+ N^k \{0,2^r\ell_1\})\cup (\{a m^k\}+ N^k\{0,2^r\ell_2\}).
$$ 
Note that $\mu_{N,\D}$ is spectral if and only if  $ \mu_{N,m^k\D}$ is spectral.

\medskip
We now prove that $m^k\D$ is a one-stage product-form  of $N$ generated by some Hadamard triples. Inded, let $\A = \{0,am^k\}$, $\B_0 = \{0,2^r\ell_1\}$ and $\B_a = \{0,
2^r\ell_2\}$. Define  ${\mathcal L}_1 = \{0, \frac{N}{2}\}$ and ${\mathcal L}_{2} = \{0, N2^{r-1}\}$. We can check immediately that $(N,\A,{\mathcal L}_1)$ and $(N,\B_i,{\mathcal L}_2)$ are Hadamard triples for all $i = 0$ or $a$.   Moreover, $(N, \A\oplus \B_i, {\mathcal L}_1\oplus{\mathcal L}_2)$ also forms a Hadamard triple (note that this is false if $r = 0$).\qquad$\Box$

\medskip

\subsection{Proof of the necessity of Theorem \ref{main theo}.} \label{subsection 4.2}We now focus on proving the necessity of Theorem \ref{main theo}. The following lemma first settle the contraction ratio. 

\medskip
\begin{Lem}\label{main-lem1}
If $\mu_{\rho, \D}$ is a spectral measure, then $\rho=\frac{1}{N}$ where $N$ is an even integer.
\end{Lem}
\begin{proof} By Lemma \ref{lem4.1}, The zero set ${\mathcal Z}(M_{\D})$ are all in rational numbers and hence they are all in a lattice. Therefore, An-Wang Theorem (Theorem \ref{lem8}) implies that $\rho = 1/N$ for some integer $N$.  We now show that $N$ cannot be odd. 

\medskip

If $N$ is odd, then $N^k {\mathbb O}\subset {\mathbb O}$.  it follows from Lemma \ref{lem4.1} that in any cases, we always have
\begin{eqnarray*}
{\mathcal Z}(\widehat{\mu}_{\rho, \D})=\bigcup_{k=1}^{\infty}N^k{\mathcal Z}(M_{\D})\subset {\mathcal Z}(M_{\D}).
\end{eqnarray*}
It implies that any bi-zero set  $\Lambda$  of $\mu_{\rho, \D}$ is also a bi-zero set of $\delta_{\D}$. So $\Lambda$ must be finite and hence can not a spectrum of $\mu_{\rho, \D}$. This contradicts to $\mu_{\rho, \D}$ is a spectral measure. So  $N$ is an even integer.
\end{proof}

\bigskip

In the following with the digit set given by (\ref{eqD_rational}) and $N = 2^{\beta}m$ where $\beta\ge 1$ and $m\in{\mathbb O}$, we will show that the two cases will not be a spectral measure. Once it proven, the necessity of Theorem \ref{main theo} will follow. 
 \begin{enumerate}
    \item $t_1\ne t_2$.
    \item $t_1 = t_2=\beta r$ for some integer $r\ge 1$. 
\end{enumerate}
For simplicity of writing, we will denote by $e_{\lambda}(x) = e^{2\pi i \lambda x}$ from now on. 
We will show that for all bi-zero set $\Lambda$ of $\mu_{N,\D}$, there exists non-zero function $F\in L^2(\mu_{N,\D})$ such that $\langle F, e_{\lambda}\rangle_{L^2(\mu)} = 0$ for all $\lambda\in \Lambda$. This shows that none of the bi-zero set is a complete set and hence there is no spectrum.

\medskip
\noindent \underline{\textbf{Case 1: {$t_1\neq t_2$}}}. According to Lemma \ref{lem4.1} (i), if we denote by $p_1$ and $p_2$ the numbers gcd$(a, c-b)$ and gcd$(c,b-a)$, which are odd integers, then we know that 
\begin{eqnarray*}
{\mathcal Z}(M_{\D})=\frac{1}{2p_1}{\mathbb O}\cup  \frac{1}{2p_2}{\mathbb O}
\end{eqnarray*}
and 
\begin{eqnarray}\label{eq-Zmu}
{\mathcal Z}(\widehat{\mu}_{\rho, \D})=\bigcup_{k=1}^{\infty}\left(\frac{N^k}{2p_1}{\mathbb O}\cup  \frac{N^k}{2p_2}{\mathbb O}\right). 
\end{eqnarray}
Let us denote also ${\mathcal Z}_k =\frac{N^k}{2p_1}{\mathbb O}\cup  \frac{N^k}{2p_2}{\mathbb O}$. When $k = 1$, we define ${\mathcal R}_1 = \frac{N}{2p_1}{\mathbb O}$ and ${\mathcal R}_2 = \frac{N}{2p_2}{\mathbb O}$.  We need some lemmas. 

\medskip

\begin{Lem}\label{lem-lam}
 Let $0\in\Lambda$ be a bi-zero set of $\mu$ and suppose that $\Lambda_1 : = \Lambda\cap {\mathcal Z}_1\ne \emptyset$.  then  the following two implications hold.
 $$
 \Lambda_1\cap ({\mathcal R}_1 \ \setminus \ {\mathcal R}_2)\ne\emptyset \ \Longrightarrow \ \Lambda_1\cap({\mathcal R}_2 \ \setminus  \ {\mathcal R}_1)=\emptyset .
 $$
 $$
 \Lambda_1\cap ({\mathcal R}_2 \ \setminus \ {\mathcal R}_1)\ne\emptyset \ \Longrightarrow \ \Lambda_1\cap({\mathcal R}_1 \ \setminus  \ {\mathcal R}_2)=\emptyset .
 $$
 Consequently, $\Lambda_1\subset {\mathcal R}_1$ or $\Lambda_1\subset {\mathcal R}_2$.

\end{Lem}
\begin{proof} By symmetry of $i = 1$ and $2$, we only need to consider one of the implications. Suppose the implication  is not true. Then  there are elements  $\lambda_1, \lambda_2\in\Lambda$ such that
\begin{equation}\label{eq-1}
\lambda_1=\frac{N(1+2\alpha_1)}{2p_1}\in{\mathcal R}_1 \ \setminus \ {\mathcal R}_2
\end{equation}
 and \begin{equation}\label{eq-2}
 \lambda_2=\frac{N(1+2\alpha_2)}{2p_2}\in{\mathcal R}_2 \ \setminus \ {\mathcal R}_1
 \end{equation}
  where $\alpha_1,\alpha_2\in\Z$.
  
  The orthogonality of $\Lambda$ implies that $\lambda_1-\lambda_2\in {\mathcal Z}(\widehat{\mu}_{\rho,\D})$. We will show that indeed $\lambda_1-\lambda_2\in \bigcup_{k=2}^{\infty}\left(\frac{N^k}{2p_1}{\mathbb O}\cup  \frac{N^k}{2p_2}{\mathbb O}\right)$. i.e. $k=1$ is not possible. Suppose that if $k = 1$. Then we will have an equation for $\lambda_1 =\frac{N(1_2\alpha_1)}{2p_1}$, $\lambda_2 = \frac{N(1+2\alpha_2)}{2p_2}$ with $\alpha_1,\alpha_2\in{\mathbb Z}$, 
  $$
  \frac{N(1+2\alpha_1)}{2p_1}-\frac{N(1+2\alpha_2)}{2p_2} = \frac{N(1+2\alpha_3)}{2p_j}
  $$
  where $j = 1$ or $2$ and $\alpha_3\in{\mathbb Z}$. If $j = 1$, we rearrange the equation and obtain 
  $$
  p_1 (1+2\alpha_2) = 2p_2 (\alpha_1-\alpha_3). 
  $$
  The left hand side is an odd number but the right hand side is an even number. The same applies to $j = 2.$ Therefore, $k = 1$ is not possible. Hence, we have now
$$
\frac{N(1+2\alpha_1)}{2p_1}-\frac{N(1+2\alpha_2)}{2p_2}=\frac{N^{j}(1+2\alpha)}{2p_i}
$$
for some $j\geq 2$ where $i = 1$ or $2$. When $i = 1$, 
\begin{eqnarray*}
\frac{N(1+2\alpha_2)}{2p_2}&=&\frac{N(1+2\alpha_1)}{2p_1}-\frac{N^{j}(1+2\alpha)}{2p_1}\\
&=&\frac{N}{2p_1}(1+2\alpha_1-N^{j-1}(1+2\alpha))\\
&\in&\frac{N}{2p_1}{\mathbb O}.
\end{eqnarray*}
which contradicts to \eqref{eq-2}. Similarly, if $i = 2$, we will contradict \eqref{eq-1}. 

\medskip

As we know that $\Lambda_1\subset {\mathcal R}_1\cup{\mathcal R}_2 = ( {\mathcal R}_1 \ \setminus \ {\mathcal R}_2) \cup ( {\mathcal R}_1\cap{\mathcal R}_2) \cup ( {\mathcal R}_2 \ \setminus \ {\mathcal R}_1)$. The above implication shows that either one of the relative complements  is empty  or both are empty. This means that $\Lambda_1\subset {\mathcal R}_1$ or $\Lambda_1\subset {\mathcal R}_2$ holds.  This completes the proof of the lemma. 
\end{proof}

\medskip
Let use write $K(N,\D)$ to be the attractor of the IFS generated by $\mu_{N^{-1},\D}$. We know that the self-similar measure $\mu_{N^{-1}, \D}$ is supported on $K(N,\D)$, which satisfies the invariant equation.  
$$K(N, \D)=N^{-1}\D+N^{-1}K(N, D).$$

\begin{theorem}
If $t_1\neq t_2$, then $\mu_{N, \D}$ is not a spectral measure.
\end{theorem}
\begin{proof}
Let $\Lambda$ be any bi-zero set of $\mu_{N, \D}$ and, without loss of generality, assume $0\in\Lambda$. The bi-zero property forces
$$\Lambda\setminus\{0\}\subset(\Lambda-\Lambda)\setminus\{0\}\subset\bigcup_{k=1}^{\infty}{\mathcal Z}_k.$$  Suppose $\Lambda \cap {\mathcal Z}_1\neq\emptyset$. By Lemma \ref{lem-lam}, either $\Lambda\cap {\mathcal Z}_1\subset {\mathcal R}_1$ or $\Lambda\cap {\mathcal Z}_1\subset {\mathcal R}_2$.   the two inclusions also vacuously  hold when $\Lambda \cap {\mathcal Z}_1=\emptyset$. Suppose the first case holds. Let 
$$f={\mathcal X}_{N^{-1}\{0, a\}}-{\mathcal X}_{N^{-1}\{b,c\}}\in L^{2}(\delta_{N^{-1}\D}),$$
and 
$$F=f\ast {\mathcal X}_{N^{-1}K(N, \D)}\in L^{2}(\mu_{N^{-1}, \D}).$$
Then for any $\lambda\in\Lambda$,  we have 
\begin{eqnarray*}
&&\langle F, e_{\lambda}\rangle_{L^2(\mu)}=\int_{K(N, \D)}F(x)e^{-2\pi i \lambda x}d\mu(x)\\
&=&\frac1{4}{\widehat{\mu}}(N^{-1}\lambda)\cdot\left(1+e^{-2\pi i \lambda N^{-1}a}-e^{-2\pi i \lambda N^{-1}b}-e^{-2\pi i \lambda N^{-1}c}\right).
\end{eqnarray*}
As $\mu_{N^{-1},\D} = \delta_{N^{-1}\D}\ast \mu_{>1}$ where $\mu_{>1} (\cdot) = \mu (N\cdot )$.  If $\lambda\in \Lambda\subset {\mathcal Z}(\widehat{\mu_{N^{-1},\D}})$, then we have 
$$
\lambda\in \bigcup_{k=2}^{\infty}{\mathcal Z}_k \ \mbox{or} \ {\mathcal Z}_1
$$
If $\lambda\in \bigcup_{k=2}^{\infty}{\mathcal Z}_k$, then $\widehat{\mu}(N^{-1}\lambda) = 0$. Also, the inner product is zero for $\lambda = 0$ since the trigonometric sum is zero at $\lambda = 0$. Hence, $\langle F, e_{\lambda}\rangle_{L^2(\mu)}=0$. 

\medskip

It remains to show that if $\lambda\in  {\mathcal Z}_1$, then $\langle F, e_{\lambda}\rangle_{L^2(\mu)}=0$ also. Now, we can write $\lambda = \frac{N}{2p_1}r$ where $r\in{\mathbb O}$ since it only intersects at $\frac{N}{2p_1}{\mathbb O}$. 
$$
\begin{aligned}
1+e^{-2\pi i \lambda N^{-1}a}-e^{-2\pi i \lambda N^{-1}b}-e^{-2\pi i \lambda N^{-1}c} = &1+e^{-\pi i \frac{a}{p_1}}-e^{-\pi i\frac{b}{p_1}}-e^{-2\pi i \frac{c}{p_1}}\\
 = & \left(1+e^{-\pi i \frac{a}{p_1}}\right)- e^{-\pi i \frac{b}{p_1}}\left(1+ e^{-\pi i \frac{c-b}{p_1}}\right)
\end{aligned}
$$
Recall that $p_1=\gcd(a, c-b)$ and $a, c-b, p_1$ are odd. $e^{-\pi i \frac{a}{p_1}} = e^{-\pi i \frac{c-b}{p_1}}=-1$, meaning the above trigonometric sum is zero. Hence, $\langle F, e_{\lambda}\rangle_{L^2(\mu)}=0$ as well. In the case that $\Lambda\cap {\mathcal Z}_1\subset {\mathcal R}_2$. We just let $f = {\mathcal X}_{N^{-1}\{0,c\}}-{\mathcal X}_{N^{-1}\{a,b\}}$. Recalling that $p_2 = \mbox{gcd}(c,b-a)$. The same argument applies. In all cases, for any bi-zero set $\Lambda$ of $\mu$,  we have found a non-zero $F\in L^2(\mu_{N^{-1},\D})$ such that $F$ is orthogonal to all $e_{\lambda}$. This means that none of such $E(\Lambda)$ is a complete set. We conclude the proof of the theorem. 
\end{proof}


\medskip

\noindent \underline{\textbf{Case 2: {$t_1= t_2 = \beta r$ for some integer $r\ge 1$}}}. The proof strategy is similar to the Case 1, but it is more complicated and requires a more careful writing. According to Lemma \ref{lem4.1} (ii), we denote by $p_1$, $p_2$ and $p_3$ the numbers gcd$(a, c-b)$, gcd$(c,b-a)$ and gcd$(\ell_1,\ell_2)$, which are odd integers, respectively. We know that 
\begin{eqnarray*}
{\mathcal Z}(M_{\D})=\frac{1}{2p_1}{\mathbb O}\cup  \frac{1}{2p_2}{\mathbb O}\cup \frac{1}{2^{1+t}p_3}{\mathbb O}
\end{eqnarray*}
and in the same way in case 1, 
\begin{equation}\label{eqZ_k}
{\mathcal Z}(\widehat{\mu}_{\rho, \D})=\bigcup_{k=1}^{\infty}{\mathcal Z}_k, \ \mbox{where} \ {\mathcal Z}_k = \frac{N^k}{2p_1}{\mathbb O}\cup  \frac{N^k}{2p_2}{\mathbb O}\cup \frac{N^k}{2^{1+t}p_3}{\mathbb O}.
\end{equation}
The crucial zero set we will look at is   ${\mathcal Z}_1\cup{\mathcal Z}_{r+1}$. To simplify notation, let us write 
$${\mathcal R}_1=\frac{N}{2p_1}{\mathbb O},\ {\mathcal R}_2=\frac{N}{2p_2}{\mathbb O},\  {\mathcal R}_3=\frac{N}{2^{1+t}p_3}{\mathbb O}.$$
$$
{\mathcal S}_1=\frac{N^{r+1}}{2p_1}{\mathbb O},\ {\mathcal S}_2=\frac{N^{r+1}}{2p_2}{\mathbb O},\  {\mathcal S}_3=\frac{N^{r+1}}{2^{1+t}p_3}{\mathbb O}.
$$
Hence, ${\mathcal Z}_1 = {\mathcal R}_1\cup {\mathcal R}_2\cup{\mathcal R}_3$ and ${\mathcal Z}_{r+1} = {\mathcal S}_1\cup {\mathcal S}_2\cup{\mathcal S}_3$.  As $t = \beta r$ and $N = 2^{\beta}m$, we have that 
$$
{\mathcal S}_3 = \frac{N}{2p_3} (m^r{\mathbb O}) \subset \frac{N}{2p_3} {\mathbb O}
$$
which has a similar structure as ${\mathcal R}_1$ and ${\mathcal R}_2$. 
\medskip

\begin{Lem}\label{lemma4.6}
Suppose $t_1=t_2=\beta r$ for some integer $r\geq1$ and $\gcd(\D)=1$, then $\gcd(p_1,p_2) = \gcd(p_1,p_3) =\gcd(p_2,p_3)=1$.
\end{Lem}

\begin{proof}
Recall that $p_1=\gcd(a, c-b)$, $p_2=\gcd(c, b-a)$ and $2^tp_3=\gcd(b, c-a)$. Let $g_{1,2} = \gcd(p_1, p_2)$. Then $g_{1,2}$ divides $a, c-b$ and $c$ by the definition of $p_1$ and $p_2$. As $\gcd(\D)=1$, it forces that $g_{1,2} =\gcd(p_1,p_2) =1$. The proof of other cases are similar.
\end{proof}


\bigskip

\begin{Lem}\label{lem-lam2}
Suppose $t_1= t_2=\beta r$  for some integer $r\ge 1$. Let $0\in\Lambda$ be a bi-zero set of $\mu$, then
$$
\Lambda\cap \left(\left({\mathcal R}_1\cup {\mathcal R}_2 \right) \ \setminus  \ {\mathcal S}_3\right)
 =\emptyset \
\mbox{\rm  or } \ 
\Lambda\cap \left({\mathcal S}_3 \ \setminus \ \left({\mathcal R}_1\cup {\mathcal R}_2 \right) \right)
=\emptyset.
$$
\end{Lem}
\begin{proof} The lemma is trivial if $\Lambda$ does not intersect ${\mathcal R}_1\cup{\mathcal R}_2\cup{\mathcal S}_3$. We would therefore assume this intersection is non-empty.  Suppose the  conclusion is not true and  there are elements  $\lambda_1\in \Lambda\cap \left({\mathcal R}_1\cup {\mathcal R}_2 \right) \setminus {\mathcal S}_3$ and $\lambda_2\in\Lambda\cap{\mathcal S}_3\setminus \left({\mathcal R}_1\cup {\mathcal R}_2 \right) $. Writing $\lambda_1=\frac{N}{2p_i}(1+2\alpha_1)$ where $i=1$ or $2$ and $\lambda_2=\frac{N^{r+1}}{2^{1+t}p_3}(1+2\alpha_2)$ with $\alpha_1,\alpha_2\in\Z$. Suppose $\lambda_1-\lambda_2\in \bigcup_{k=1}^{\infty} {\mathcal Z}_k$. By (\ref{eqZ_k}), we have two cases.

\medskip

\noindent (i). Suppose that $\lambda_1-\lambda_2\in \frac{N^k}{2p_1}{\mathbb O}\cup\frac{N^k}{2p_2}{\mathbb O}$. Then for some $j = 1 $ or $2$, 
$$
\frac{N}{2p_i}(1+2\alpha_1)- \frac{N^{r+1}}{2^{1+t}p_3}(1+2\alpha_2) = \frac{N^k}{2p_j}(1+2\alpha').
$$
Rearranging the equation into integers, 
\begin{equation}\label{eq_rearrange1}
p_3\left(p_j(1+2\alpha_1)- N^{k-1}p_3p_i(1+2\alpha')\right) =p_ip_jm^r(1+2\alpha_2). 
\end{equation}
Examining the parity, it is only possible that $k>1$.  By Lemma \ref{lemma4.6}, $p_3$ divides $m^{r}(1+2\alpha_2)$. We now have 
$$
\lambda_2 = \frac{N^{r+1}}{2^{1+t}p_3}(1+2\alpha_2) = \frac{N}{2} \left(\frac{m^r(1+2\alpha_2)}{p_3}\right).
$$
As $p_3$ divides $m^{r}(1+2\alpha_2)$, the number in the parenthesis on the right is an odd integer, meaning that it belongs to ${\mathcal R}_1$, which is a contradiction since we picked  $\lambda_2\in {\mathcal S}_3\setminus ({\mathcal R}_1\cup {\mathcal R}_2)$. 

\medskip

\noindent (ii) Suppose now that  $\lambda_1-\lambda_2\in \frac{N^k}{2^{1+t}p_3}{\mathbb O}$, then 
\begin{equation}\label{eq4.7.4}
\frac{N}{2p_i}(1+2\alpha_1)- \frac{N^{1+r}}{2^{1+t}p_3}(1+2\alpha_2) = \frac{N^k}{2^{1+t}p_3}(1+2\alpha').
\end{equation}
Rearranging the equation into integers, 
\begin{equation}\label{eq_rearrange1}
p_3\left(1+2\alpha_1\right) =p_i\left(m^r(1+2\alpha_2)+2^{\beta(k-1-r)}m^{k-1}(1+2\alpha')\right). 
\end{equation}
Examining the parity, it is only possible that $k>r+1$. Then we can rewrite \eqref{eq4.7.4} as  
$$
\lambda_1 = \frac{N^{r+1}}{2^{1+t}p_3}(1+2\alpha_2) +\frac{N^k}{2^{1+t}p_3}(1+2\alpha')=\frac{N^{r+1}}{2^{1+t}p_3}\left(1+2\alpha_2+N^{k-r-1}(1+2\alpha')\right).
$$ 
As $k>r+1$, the number in the parenthesis on the right is an odd integer, meaning that it belongs to ${\mathcal S}_3$, which is a contradiction since we picked  $\lambda_1\in ({\mathcal R}_1\cup{\mathcal R}_2) \setminus {\mathcal S}_3$.  Hence, the lemma is proven. 
\end{proof}

\medskip

\begin{Lem}\label{lem-lam3}
Suppose $t_1= t_2=\beta r$  for some integer $r\ge 1$. Let $0\in\Lambda$ be a bi-zero set of $\mu$, then 
$$\Lambda\cap\left({\mathcal S}_1 \ \setminus   \ {\mathcal S}_2\right)=\emptyset \text{ or }\Lambda\cap\left({\mathcal S}_2 \ \setminus   \ {\mathcal S}_1\right)=\emptyset.$$
\end{Lem}

\medskip

\begin{proof} Similar to Lemma \ref{lem-lam2}, the lemma is trivial if $\Lambda\cap ({\mathcal S}_1\cup{\mathcal S}_2) = \emptyset$. Hence, we assume the intersection is non-empty. Suppose the conclusion is not true. Then  there are elements  $\lambda_1, \lambda_2\in\Lambda$ such that
\begin{equation*}\label{eq-1}
\lambda_1=\frac{N^{r+1}}{2p_1}(1+2\alpha_1)\in {\mathcal S}_1 \ \setminus   \ {\mathcal S}_2
\end{equation*}
 and \begin{equation*}\label{eq-2}
 \lambda_2=\frac{N^{r+1}}{2p_2}(1+2\alpha_2)\in {\mathcal S}_2 \ \setminus   \ {\mathcal S}_1
 \end{equation*}
  where $\alpha_1,\alpha_2\in\Z$.  Using the bizero set property, we know that $\lambda_1-\lambda_2\in\bigcup_{k = 1}^{\infty} {\mathcal Z}_k$.  If $\lambda_1-\lambda_2\in \frac{N^k}{2p_1}{\mathbb O}$, then 
\begin{equation}\label{eq4.7.5}
\frac{N^{r+1}}{2p_1}(1+2\alpha_1)- \frac{N^{1+r}}{2p_2}(1+2\alpha_2) = \frac{N^k}{2p_1}(1+2\alpha').
\end{equation}
After some arrangement, we have 
\begin{equation}\label{eq_rearrange1}
p_1\left(1+2\alpha_2\right) =p_2\left(1+2\alpha_2-2^{\beta(k-1-r)}m^{k-r-1}(1+2\alpha')\right). 
\end{equation}
Examining the parity, it is only possible that $k>r+1$. Then we can rewrite \eqref{eq4.7.5} as  
$$
\lambda_2 = \frac{N^{r+1}}{2p_1}(1+2\alpha_2) -\frac{N^k}{2p_1}(1+2\alpha')=\frac{N^{r+1}}{2p_1}\left(1+2\alpha_2-N^{k-r-1}(1+2\alpha')\right).
$$ 
As $k>r+1$, the number in the parenthesis on the right is an odd integer, meaning that it belongs to $\frac{N^{r+1}}{2p_1}{\mathbb O}$, which is a contradiction since we picked  $\lambda_2\in {\mathcal S}_2\setminus{\mathcal S_1}$. The same applied to the case that $\lambda_1-\lambda_2\in \frac{N^k}{2p_2}{\mathbb O}$.

\bigskip

Suppose now that 
\begin{equation}\label{eq4.7.6}
\frac{N^{r+1}}{2p_1}(1+2\alpha_1)- \frac{N^{r+1}}{2p_2}(1+2\alpha_2) = \frac{N^k}{2^{1+t}p_3}(1+2\alpha')\in \frac{N^k}{2^{1+t}p_3}{\mathbb O}.
\end{equation}
We now recall that $t = \beta r$, $N = 2^{\beta}m$ and rearrange the equation
$$
 p_2p_3(1+2{\alpha}_1)=p_1 \left(p_3(1+2\alpha_2)+2^{\beta(k-r-1)}m^{k-r-1}p_2(1+2\alpha')\right).
$$
Examining the parity again, it is only possible that $k>r+1$.
By Lemma \ref{lemma4.6}, $p_1$ divides $1+2\alpha_1$. We now write (\ref{eq4.7.6}) as 
$$
\lambda_1 = \frac{N^{r+1}}{2p_2}\left(\frac{p_2(1+2\alpha_1)}{p_1}\right). 
$$
As $p_1$ divides $1+2\alpha_1$, the number in the parenthesis on the right is an odd integer, meaning that it belongs to ${\mathcal S}_2$, which is a contradiction since we picked  $\lambda_1\in {\mathcal S}_1\setminus{\mathcal S}_2$.  We have thus proved the lemma.
\end{proof}

\bigskip

\begin{Lem}\label{lem-case2}
Suppose that $\Lambda$ is a bi-zero set. Consider  the following statements:
\begin{enumerate}
\item[(I)] $\Lambda\cap\left({\mathcal Z}_1 \ \setminus \ \bigcup_{k=2}^{\infty}{\mathcal Z}_k\right)\subset {\mathcal R}_3$;
    \item[(II)]$\Lambda\cap\left({\mathcal Z}_{r+1} \ \setminus \ \bigcup_{k\neq r+1}{\mathcal Z}_k\right)\subset {\mathcal S}_1$; 
    \item [(III)] $\Lambda\cap\left({\mathcal Z}_{r+1} \ \setminus \ \bigcup_{k\neq r+1}{\mathcal Z}_k\right)\subset {\mathcal S}_2$.
\end{enumerate}
Then the statement $(I)$ or $(II)$ or $(III)$ holds. 
\end{Lem}

\begin{proof} We first note that if $\Lambda\cap ({\mathcal Z}_1\cup{\mathcal Z}_{r+1}) = \emptyset$, then the statements (I)-(III) are all true vacuously as the left hand side are all empty sets.  Hence, we may assume that $\Lambda\cap ({\mathcal Z}_1\cup{\mathcal Z}_{r+1}) \ne \emptyset$. Note that  ${\mathcal Z}_1\cup{\mathcal Z}_{r+1}$ is a union of 6 sets.    If $\Lambda\cap ({\mathcal R}_1\cup{\mathcal R}_2) = \emptyset$, then (I) holds because the left hand side of (I) is a subset of
$$
\Lambda \cap {\mathcal Z}_1 =\Lambda\cap\left({\mathcal R}_1\cup{\mathcal R}_2\cup{\mathcal R}_3\right)= \Lambda\cap{\mathcal R}_3\subset {\mathcal R}_3.
$$
Similarly, if $\Lambda\cap {\mathcal S}_3 = \emptyset$, then $\Lambda\cap {\mathcal Z}_{r+1} = \Lambda\cap ({\mathcal S}_1\cup{\mathcal S}_2)$. By Lemma \ref{lem-lam3}, $\Lambda\cap {\mathcal Z}_{r+1}\subset {\mathcal S}_1$ or $\Lambda\cap {\mathcal Z}_{r+1}\subset {\mathcal S}_2$. Therefore, (II) or (III) holds. 

\medskip

It remains to see the case that $\Lambda\cap ({\mathcal R}_1\cup{\mathcal R}_2\cup{\mathcal S}_3)\ne \emptyset.$ We then apply Lemma \ref{lem-lam2} implies that 
$$
\Lambda\cap (({\mathcal R}_1\cup{\mathcal R}_2) \ \setminus  \ {\mathcal S}_3) =  \emptyset \  \mbox{or} \  \Lambda\cap ({\mathcal S}_3 \ \setminus  \ ({\mathcal R}_1\cup{\mathcal R}_2) )=  \emptyset.
$$
\medskip
Suppose that 
$$
\Lambda\cap \left({\mathcal R}_1\cup {\mathcal R}_2 \right) \setminus {\mathcal S}_3=\emptyset.
$$
Then 
$$
\Lambda\cap\left({\mathcal Z}_1 \ \setminus \ \bigcup_{k=2}^{\infty}{\mathcal Z}_k\right)\subset\Lambda\cap \left(({\mathcal R_3}\cup {\mathcal R}_1\cup {\mathcal R}_2)  \ \setminus \ {\mathcal S}_3 \right)\subset{\mathcal R}_3.
$$
If we have now that 
$\Lambda\cap \left({\mathcal S}_3 \ \setminus\  \left({\mathcal R}_1\cup {\mathcal R}_2 \right)\right) =\emptyset,
$
then 
\begin{eqnarray*}
\Lambda\cap\left({\mathcal Z}_{r+1} \ \setminus \ \bigcup_{k\ne r+1}{\mathcal Z}_k\right)&\subset&\Lambda\cap\left(\left({\mathcal S}_1\cup{\mathcal S}_2\cup{\mathcal S}_3\right)\setminus ({\mathcal R}_1\cup {\mathcal R}_2)   \right)\\
&\subset&\Lambda\cap\left({\mathcal S}_1\cup{\mathcal S}_2\right).
\end{eqnarray*}
Using Lemma \ref{lem-lam3}, (II) or (III) holds. This completes the proof. 
\end{proof}


\bigskip

\bigskip
We know that the self-similar measure $\mu_{N^{-1}, \D}$ is supported on $K(N,\D)$, which satisfies the equations.
$$K(N, \D)=N^{-1}\D+N^{-1}K(N, D)=N^{-r-1}\D+\widetilde{K}$$
where $\widetilde{K}=\sum_{k\neq r+1}N^{-k}\D.$
\bigskip

\begin{theorem}
Suppose $t=t'=\beta r$ for some integer $r\geq1$, then $\mu_{N^{-1},\D}$ is not a spectral measure.
\end{theorem}
\begin{proof} Let $\Lambda$ be any bi-zero set of $\mu_{N, \D}$ and, without loss of generality, assume $0\in\Lambda$. The bi-zero property forces
$$\Lambda\setminus\{0\}\subset(\Lambda-\Lambda)\setminus\{0\}\subset\bigcup_{k=1}^{\infty}{\mathcal Z}_k.$$
By Lemma \ref{lem-case2},  we have (I) or (II) or (III) holds. 

\bigskip

  Suppose that (I) holds, i.e.  $\Lambda\cap\left({\mathcal Z}_1 \ \setminus \ \bigcup_{k=2}^{\infty}{\mathcal Z}_k\right)\subset {\mathcal R}_3$.  Let 
$$f={\mathcal X}_{N^{-1}\{0, b\}}-{\mathcal X}_{N^{-1}\{a,c\}}\in L^{2}(\delta_{N^{-1}\D}),$$
and 
$$F=f\ast {\mathcal X}_{N^{-1}K(N, \D)}\in L^{2}(\mu_{N^{-1}, \D}).$$
Then for any $\lambda\in\Lambda$,  we have 
\begin{eqnarray*}
&&\langle F, e_{\lambda}\rangle_{L^2(\mu)}=\int_{K(N, \D)}F(x)e^{-2\pi i \lambda x}d\mu(x)\\
&=&\frac1{4}{\widehat{\mu}}(N^{-1}\lambda)\cdot\left(1+e^{-2\pi i \lambda N^{-1}a}-e^{-2\pi i \lambda N^{-1}b}-e^{-2\pi i \lambda N^{-1}c}\right).
\end{eqnarray*}
As $\mu_{N^{-1},\D} = \delta_{N^{-1}\D}\ast \mu_{>1}$ where $\mu_{>1} (\cdot) = \mu (N\cdot )$.  If $\lambda\in \Lambda\subset {\mathcal Z}(\widehat{\mu_{N^{-1},\D}})$, then we have 
$$
\lambda\in \bigcup_{k=2}^{\infty}{\mathcal Z}_k \ \mbox{or} \ {\mathcal Z}_1 \ \setminus \ \bigcup_{k=2}^{\infty}{\mathcal Z}_k
$$
If $\lambda\in \bigcup_{k=2}^{\infty}{\mathcal Z}_k$, then $\widehat{\mu}(N^{-1}\lambda) = 0$. Also, the inner product is zero for $\lambda = 0$ since the trigonometric sum is zero at $\lambda = 0$. Hence, $\langle F, e_{\lambda}\rangle_{L^2(\mu)}=0$. 

\medskip

It remains to show that if $\lambda\in  {\mathcal Z}_1 \ \setminus \ \bigcup_{k=2}^{\infty}{\mathcal Z}_k$, then $\langle F, e_{\lambda}\rangle_{L^2(\mu)}=0$ also. Now, we can write $\lambda = \frac{N}{2^{1+t}p_3}r$ where $r\in{\mathbb O}$ since it is a subset of ${\mathcal R}_3$ by Lemma \ref{lem-case2}. We have
$$
\begin{aligned}
1+e^{-2\pi i \lambda N^{-1}b}-e^{-2\pi i \lambda N^{-1}a}-e^{-2\pi i \lambda N^{-1}c} = &1+e^{-\pi i \frac{a}{2^{1+t}p_3}}-e^{-\pi i\frac{b}{2^{1+t}p_3}}-e^{-2\pi i \frac{c}{2^{1+t}p_3}}\\
 = & \left(1+e^{-\pi i \frac{b}{2^{1+t}p_3}}\right)- e^{-\pi i \frac{a}{2^{1+t}p_3}}\left(1+ e^{-\pi i \frac{c-a}{2^{1+t}p_3}}\right)
\end{aligned}
$$
Recall that $b=2^t\ell, c=a+2^t\ell'$ and $p_3=\gcd(\ell, \ell')$ where  $\ell, \ell', p_3$ are odd. $e^{-\pi i \frac{b}{2^{1+t}p_3}} = e^{-\pi i \frac{c-a}{2^{1+t}p_3}}=-1$, meaning the above trigonometric sum is zero. Hence, $\langle F, e_{\lambda}\rangle_{L^2(\mu)}=0$ as well. 

\bigskip

Suppose that (II) or (III) holds.   In this case we consider $$f_2 = {\mathcal X}_{N^{-r-1}\{0,a\}}-{\mathcal X}_{N^{-r-1}\{b,c\}}, \quad \quad f_3 = {\mathcal X}_{N^{-r-1}\{0,c\}}-{\mathcal X}_{N^{-r-1}\{a,b\}}$$ 
for (II) and (III) respectively. The same argument applies to $F=f\ast{\mathcal X}_{\widetilde{K}}$. In all cases, for any bi-zero set $\Lambda$ of $\mu$,  we have found a non-zero $F\in L^2(\mu_{N^{-1},\D})$ such that $F$ is orthogonal to all $e_{\lambda}$. This means that none of such $E(\Lambda)$ is a complete set. We conclude the proof of the theorem.
\end{proof}

\section{Concluding Remark}

This paper offered a complete characterization of the spectral self-similar measures generated by four similitudes. The case for two similitudes  was  first solved by Dai \cite{D2012} and most of the later arguments were developed from Dai's paper and also \cite{DHL2014}. Nonetheless, with all theorems we introduced in this paper, we can now easily provide  a  complete characterization of the spectral self-similar measures with two or three similitudes

\begin{theorem}
Suppose that $\mu_{\rho,\D,{\bf p}}$ is a spectral self-similar measure with $0\in\D\subset\R^+$ and $\#\D = 2$ or $3$. Then the {\L}aba-Wang conjecture holds: 
  \begin{enumerate}
    \item ${\bf p}$ is a equal-weight probability vector.
    \item $\rho = 1/N$ for some integer $N$.
    \item There exists $\alpha>0$ such that $\D = \alpha {\mathcal C}$ where ${\mathcal C}\subset \Z^+$ and ${\mathcal C}$ tiles $\Z_N$. i.e. there exists $\B$ such that ${\mathcal C}\oplus\B\equiv \Z_N$ (mod N).
    \end{enumerate}
\end{theorem}

\begin{proof}  (i) follows from Theorem \ref{Theorem_DC}. Next we will prove condition (ii) and (iii) in the following two cases.

 When $\D=\{0, a\}$ has two elements, we let ${\mathcal C}=\{0, 1\}$. Then  
$\D=a{\mathcal C}$ and ${\mu}_{\rho, {\mathcal C}}$ is also a spectral measure. In this case, ${\mathcal Z}(M_{{\mathcal C}})=\frac1{2}{\mathbb O}$ and 
$${\mathcal Z}(\widehat{\mu}_{\rho, {\mathcal C}})=\bigcup_{k=1}^{\infty}\frac{\rho^{-k}}{2}{\mathbb O}$$
has a zero set of Bernoulli structure. We immediately know that it is necessary that  $\rho=\frac{1}{N}$ with $N$ is even. So ${\mathcal C}$ tiles $\Z_N$. This proves (ii) and (iii).

When $\D=\{0, a, b\}$ has three elements, we have 
$$M_{\{0, a, b\}}(\xi)=\frac13(1+e^{2\pi i a\xi}+e^{2\pi i b\xi})=0  \Longleftrightarrow  \{a\xi, b\xi\}=\left\{\frac13+k_1, \frac23+k_2\right\}$$
for some integers $k_1, k_2$.  Then 
$$\frac{a}{b}=\frac{1+3k_1}{2+3k_2}\quad \text{or }\quad\frac{2+3k_2}{1+3k_1}$$ 
is a rational number.  We write it $\frac{a}{b}=\frac{a_1}{b_1}$ to be the simplest form, that is $a_1, b_1$ are co-prime positive integers. Let ${\mathcal C}=\{0, a_1, b_1\}$.  Then $\D=\frac{b}{b_1}{\mathcal C}$ and 
\begin{equation}\label{eq-C}
    {\mathcal C}\equiv\{0, 1, 2\}\pmod3.
\end{equation}
Since $\mu_{\rho, {\mathcal C}}$ is also a spectral measure and 
$${\mathcal Z}(M_{{\mathcal C}})=\frac1{3}(\Z\setminus3\Z),$$
by  An-Wang Theorem (Theorem \ref{lem8}), $\rho = 1/N$ for some integer $N$. This proves (ii). If $N$ is not divisible by $3$, then 
$${\mathcal Z}(\widehat{\mu}_{\rho, {\mathcal C}})=\bigcup_{k=1}^{\infty}\frac{N^k}{3}(\Z\setminus3\Z)\subset\frac13(\Z\setminus3\Z).$$
It implies that any bi-zero set  $\Lambda$  of $\mu_{\rho, {\mathcal C}}$ is also a bi-zero set of $\delta_{{\mathcal C}}$. So $\Lambda$ must be finite and hence can not a spectrum of $\mu_{\rho, {\mathcal C}}$. This contradicts to $\mu_{\rho, {\mathcal C}}$ is a spectral measure. So  $N$ is divisible by $3$. From \eqref{eq-C}, ${\mathcal C}$ tiles $\Z_N$.
\end{proof}

We see that the equation as vanishing sum of complex numbers of modulus one plays an important beginning role to set out our proofs. This set of equation can be completely solved only when there are at most four complex numbers. 

\medskip
The (modified) {\L}aba-Wang conjecture becomes a difficult problem to study if  the number of  digits is more than four because some irrationality phenomenon began to exist when the number of digits is five or more. 
The following example provides a case with purely irrational roots with five digits. It can be found in \cite{AF2019}.

\begin{example}
(See \cite[Example 7, p.1207]{AF2019}) Let $\D = \{0,1,3,5,6\}$. Then $M_{\D}(\xi)$ consists only of irrational roots for $\xi$. By periodicity, 
$$
{\mathcal Z}(M_{\D}) = \Theta+\Z
$$
where $\Theta\subset \R\setminus {\mathbb Q}$. In this case, an easy argument shows that  there does not exist any $\alpha>0$ such that  ${\mathcal Z}(M_{\D}) \subset \alpha\Z.$ Therefore, the assumption for Theorem \ref{lem8} is not satisfied. Hence, we cannot conclude immediately that $\mu_{\rho,\D}$ is spectral measure implies that $\rho^{-1}$ is an integer. We believe that none of such $\mu_{\rho,\D}$ can be a spectral measure, but a more intensive study will be needed in this case. 
\end{example}

In the end, we also remark that the digit cardinality equal to 4 or less has also been studied in other cases of singular measures, including Moran measures, and Sierpinski measures, one may refer to the papers \cite{AHL2015, DFY2021, FW2017, LDL2022} for more details. 

\medskip

\noindent{\bf Acknowledgment.} This research field of fractal spectral measures was first brought to Chun-Kit Lai when he was a PhD student of Professor Ka-Sing Lau in 2007-2012. At the same time,  Xinggang He, as a former PhD student of Ka-Sing, joined in the research project while visiting The Chinese University of Hong Kong back in 2010. Lixiang An joined the research team while she was a PhD student of Xinggang and later she became a Postdoctoral Fellow of Ka-Sing. 

\medskip

We are indebted to Professor Ka-Sing Lau's deep insight into this problem and his constant encouragement to all of us over many years.  We were sad to hear the passing of Ka-Sing. This paper is dedicated to the memory of Ka-Sing and particularly to his pioneering and fundamental contribution to the theory of fractal spectral measures. Ka-Sing was a great mentor, caring teacher and sincere friend. He will be dearly missed by all of us.













\end{document}